\documentclass[12pt,oneside,reqno]{amsart}
\usepackage{makecell,float}
\usepackage{mathrsfs}
\usepackage{txfonts}
\allowdisplaybreaks
\sloppypar

\usepackage{color}

\usepackage[colorlinks=true,backref=page]{hyperref}
\hypersetup{
    linkcolor=blue,          % color of internal links
    citecolor=red,        % color of links to bibliography
    filecolor=blue,      % color of file links
    urlcolor=magenta
}

\usepackage{geometry}
\numberwithin{equation}{section}
\usepackage{thmtools}
\newtheorem{theorem}{Theorem}[section]
\newtheorem{lemma}[theorem]{Lemma}
\newtheorem{proposition}[theorem]{Proposition}
\newtheorem{corollary}[theorem]{Corollary}

\newtheorem{definition}[theorem]{Definition}

\newtheorem{remark}[theorem]{Remark}
\newtheorem{example}[theorem]{Example}

\usepackage[abbrev]{amsrefs}
\renewcommand{\eprint}[1]{{\it Available at}\href{https://arxiv.org/abs/#1}{\it{ arXiv:#1}}.}
\renewcommand{\PrintDOI}[1]{\url{https://doi.org/#1}}%{#1}}
\renewcommand{\MR}[1]{\href{https://mathscinet.ams.org/mathscinet-getitem?mr=#1}{\color{cyan}{MR#1}}}

\BibSpec{article}{%
+{}{\PrintAuthors} {author}
+{,}{ \textrm} {title}
+{:}{ \textrm} {subtitle}
+{.}{ \textit} {journal}
+{,}{ \textbf} {volume}
+{}{ \parenthesize} {date}
+{,}{ } {pages}
+{}{ \parenthesize} {language}
+{.}{} {transition}
+{.}{ \PrintReviews} {review}
+{.}{ \eprint} {eprint}
+{.}{ \PrintDOI} {doi}
+{.}{ \url} {url}
}
\BibSpec{book}{%
+{}{\PrintAuthors} {author}
+{,}{ \textit} {title}
+{:}{ \textit} {subtitle}
+{.}{ \textrm} {series} %+{,}{ \textrm} {series}
+{,}{ Vol.} {volume}
+{.}{ } {publisher}
+{,}{ } {date}
+{}{ \parenthesize} {language}
+{.}{} {transition}
+{.}{ \PrintReviews} {review}
+{.}{ \PrintDOI} {doi}
}

\newcommand\bt{\begin{theorem}}
\newcommand\et{\end{theorem}}
\newcommand\bl{\begin{lemma}}
\newcommand\el{\end{lemma}}
\newcommand\bd{\begin{definition}}
\newcommand\ed{\end{definition}}
\newcommand\bp{\begin{proposition}}
\newcommand\ep{\end{proposition}}
\newcommand\bc{\begin{corollary}}
\newcommand\ec{\end{corollary}}
\newcommand\br{\begin{remark}}
\newcommand\er{\end{remark}}
\newcommand\bexa{\begin{example}}
\newcommand\eexa{\end{example}}
\newcommand\1{{\boldsymbol{1}}}
\newcommand\hp{{\boldsymbol{p}}}
\newcommand\cC{\mathcal{C}}
\newcommand{\R}{{\mathbb R}}
\newcommand{\EX}{{\mathbb{E}}}
\newcommand\mE{{\mathbb E}}
\newcommand\mI{{\mathbb I}}
\newcommand\mN{{\mathbb N}}

\newcommand\mR{{\mathbb R}}
\newcommand\sD{\mathscr{D}}
\newcommand\sL{{\mathscr L}}

\newcommand\D{\mathord{\mathrm{D}}}
\newcommand\e{\mathrm{e}}

\newcommand\dif{{\mathord{{\mathrm d}}}}
\newcommand\dis{{\boldsymbol{d}}}
\newcommand\p{\partial}

\renewcommand\[{{\Big[}}
\renewcommand\]{{\Big]}}
\newcommand\<{{\langle}}
\renewcommand\>{{\rangle}}
\renewcommand\({{\Big(}}
\renewcommand\){{\Big)}}

\newcommand\var{{\mathrm{var}}}

\begin{document}

\title{The Non-Gaussian to Gaussian Transition:  Pointwise Heat Kernel Estimates and Optimal Convergence Rates}

\author{Xianming Liu, Chongyang Ren, and Mingyan Wu}
%\date{\today}

\address{Xianming Liu: School of Mathematics and Statistics, Huazhong University of Science and Technology, Wuhan, Hubei 430074, P. R. China, Email: xmliu@hust.edu.cn}

\address{Chongyan Ren: School of Mathematical Sciences, University of Science and Technology of China, Hefei, Anhui 230026, P. R. China, Email: rcy.math@ustc.edu.cn}

\address{Mingyan Wu: School of Mathematical Sciences, Xiamen University, Xiamen, Fujian 361005, P. R. China, Email:  mingyanwu.math@gmail.com, mingyanwu.math@xmu.edu.cn }

%\thanks{* Corresponding author}



\maketitle

\begin{abstract}
We establish uniform pointwise estimates for the densities of a family of $\alpha$-stable processes with respect to the index $\alpha \in [\alpha_0,2]$ for some $\alpha_0>0$. In addition, we estimate the difference between the heat kernels of non-local and local operators, showing that it is controlled by the rate $2-\alpha$. Both estimates (see  Proposition \ref{prop:1.6}) are new to the literature. Furthermore, as an application, we achieve the optimal rate $2-\alpha$ for the pointwise estimate between the transition probabilities, as well as for the (weighted) total variation and Kantorovich distances between the invariant measures, of non-Gaussian and Gaussian diffusion. These results are obtained under the assumption that the drifts are locally $\beta$-H\"older continuous, with the latter additionally requiring dissipativity. The results on transition probabilities (see Theorem \ref{th-1.1}) are novel, while those on invariant measures (see Theorem \ref{th-6.1}) significantly extend the existing literature.
\end{abstract}

\setcounter{tocdepth}{1}
\tableofcontents

%\listoftodos{}

\section{Introduction}

\subsection{Motivation}

A L\'evy process $(L_t)_{t \geq 0}$ in $\mathbb{R}^d$ is called $\alpha$-stable if it is self-similar with stability index $\alpha \in (0,2]$; that is, for any $c>0$, the scaled process $(c^{-1/\alpha} L_{ct})_{t \geq 0}$ has the same distribution as $(L_t)_{t \geq 0}$. Here $\alpha$ is called the stability index (which also serves as the tail index: the smaller $\alpha$, the heavier the tail). When $\alpha=2$, the process is (up to a scaling factor) a standard Brownian motion, which is the only Gaussian L\'evy process. When $\alpha \in (0,2)$, the process is recognized as a non-Gaussian process.

\medskip

It is well known that when $\alpha$ is close to $2$, an $\alpha$-stable process behaves similarly to a Brownian motion, which is clearly observable in computer simulations of sample paths (see, e.g.,~\cite{JW94}). Roughly speaking, as $\alpha$ approaches $2$, the motion of an $\alpha$-stable process is dominated by small jumps (see also \cite{DZ16,SZ20}).  The limit $\alpha \to 2$ thus connects two different types of stochastic models: from non-Gaussian to Gaussian behavior.

\medskip

In recent years, SDEs driven by $\alpha$-stable processes have been studied across many fields, where the stability index $\alpha$ governs the shift from heavy-tailed, jump-type models to Gaussian continuous models. In physics, this shift characterizes the evolution from anomalous diffusion to the classical equilibrium of Brownian motion. 
The jumped processes
also appear in models of collisions in certain physical systems, such as Boltzmann equation (see e.g., \cite{Ta78}), and the Surface Quasi-Geostrophic system (see e.g., \cite{HRZ24}).  In engineering, non-Gaussian $\alpha$-stable distributions are used to characterize signals that exhibit sharp spikes or occasional bursts, where the transition $\alpha \to 2$ corresponds to a return to Gaussian signals (see e.g., \cite{SCQ}). In economics, heavy-tailed distributions—such as Pareto's law for income and wealth (see e.g., \cite{Ar15})—are important in extreme value theory and risk management (see e.g., \cite{EKM}); $\alpha$-stable processes provide a flexible family for describing such tail behavior.  In financial mathematics, the heavy tails and jumps of asset returns are well captured by $\alpha$-stable distributions. A stability index $\alpha < 2$ reflects their heavy-tailed nature and ``black swan'' events (see e.g., \cite{No}), while $\alpha = 2$ reduces to the Gaussian case, corresponding to the classical Black-Scholes framework (see e.g., \cite{Ok03}).

\medskip

Although $\alpha$-stable processes appear in many areas, a quantitative understanding of the limit $\alpha \to 2$ has received little attention in the literature. In this paper, we study the ``distance'' (in some suitable sense) between the following two stochastic differential equations (SDEs):
\begin{equation}\label{e1}
\dif X^{(\alpha)}_t = b(X^{(\alpha)}_t) \dif t + \dif L^{(\alpha)}_t,\ \ \alpha \in (0,2),
\end{equation}
and
\begin{equation}\label{e2}
\dif X^{(2)}_t = b(X^{(2)}_t) \dif t + \dif L^{(2)}_t,
\end{equation}
where $b: \mathbb{R}^d \rightarrow \mathbb{R}^d$ is a Borel measurable drift coefficient, and $(L_t^{(\alpha)})_{t \geq 0}$ is the rotationally invariant $\alpha$-stable L\'evy process with $\alpha \in (0,2]$. In particular, for SDEs with drift coefficients that are of low regularity—a condition that occurs naturally in many applications—the convergence rates as $\alpha \to 2$ are not covered by existing results. In this work, we study this problem using heat kernel estimates. We obtain convergence rates for the densities of solutions and invariant measures to the corresponding SDEs. To this end, we establish uniform heat kernel estimates with respect to the stability index $\alpha$. The main results are presented in Theorems \ref{th-1.1} and \ref{th-6.1}.

\subsection{Heat kernel estimates: from non-local to local}
 
As a matter of fact, for  rotationally invariant $\alpha$-stable L\'evy processes $L^{(\alpha)}$, the limit $\alpha \to 2$ is evident when one looks at the probability densities. Notice that the characteristic function (see e.g., \cite{Sa99}) is given by
\begin{align}\label{eq:ZG02}
\mathbb{E} \mathrm{e}^{i \xi \cdot L_t^{(\alpha)}} = \mathrm{e}^{-t |\xi|^\alpha},
\end{align}
which, by Proposition 28.1 of \cite{Sa99}, derives that  $L_t^{(\alpha)}$ admits a smooth density function $p^{(\alpha)}(t,\cdot)$ given by the inverse Fourier transform
\begin{align}\label{eq:hp01}
p^{(\alpha)}(t,x) = (2\pi)^{-d} \int_{\mathbb{R}^d} \mathrm{e}^{-i x \cdot \xi} \, \mathbb{E} \mathrm{e}^{i\xi \cdot L_t^{(\alpha)}} \, d\xi, \quad \forall t>0,
\end{align}
and all partial derivatives of $p^{(\alpha)}(t,\cdot)$ decay to $0$ as $|x| \to \infty$. When $\alpha = 2$, the process reduces to Brownian motion, whose density $p^{(2)}(t,\cdot)$ corresponds to the Laplacian (a local generator). When $\alpha < 2$, the density $p^{(\alpha)}(t,\cdot)$ is associated with the fractional Laplacian (a non-local generator). By the dominated convergence theorem, we obtain the convergence
\begin{align}\label{eq:UJ00}
p^{(\alpha)}(t,x) \to p^{(2)}(t,x), \quad \text{as } \alpha \to 2.
\end{align}
This elementary fact naturally raises two questions:
\begin{itemize}
\item[(i)] What is the convergence rate for \eqref{eq:UJ00}?

\item[(ii)] Under what conditions does such convergence extend to the densities (or distributions) of SDEs driven by $\alpha$-stable processes?
\end{itemize}

Regarding the first question, to the best of our knowledge, no result is currently available in the literature. In present work, we address this gap by establishing a pointwise estimate for the difference between the heat kernels, which yields a convergence rate of $2-\alpha$ (see Proposition \ref{prop:1.6}(ii)). A key ingredient in our proof is a novel uniform estimate for the heat kernel with respect to the index $\alpha$ (see Lemma \ref{lem:ZT00}), which is of independent interest. Precisely, for $\alpha \in [\alpha_0,2]$ with some $\alpha_0 \in (0,2)$, there exists a constant $c > 0$ depending only on $d,\alpha_0$ such that for any $t > 0$,
$$
p^{(\alpha)}(t,x) \lesssim_c t^{-d/\alpha} \wedge \frac{t}{|x|^{d+\alpha}} =: \varrho_\alpha(t,x).
$$
Compared with classical results (e.g., Lemma 2.2 in \cite{CZ16}), the difference is that our constant $c$ does not depend on the index $\alpha$. See Proposition \ref{prop:1.6}(i) for a similar result for the derivatives.

\medskip

For the second question, several results concerning the convergence of distributions have already been established. For example, the first named author, Liu, studied in \cite{Liu22a} the convergence behavior, as $\alpha \to 2$, of solutions to SDEs driven by $\alpha$-stable processes with drift belonging to the usual H\"older space. As for quantitative results, we first refer to \cite{LFL23}, in which the authors gave a  look at the optimal weak convergence rate, showing it to be $2-\alpha$ for drifts that are dissipative, locally bounded, and of H\"older regularity $2+\beta$, by constructing an evolution equation describing the weak difference. Subsequently, Deng et al. \cite{DSX23} established the optimal convergence rate $2-\alpha$ for the Wasserstein-1 distance between $\operatorname{Law}(X_t^{(\alpha)}(x))$ and $\operatorname{Law}(X_t^{(2)}(y))$. Their proof uses Malliavin calculus and assumes the drift to be dissipative and third-order differentiable. This result was later extended to the case of multiplicative noises in the master's thesis \cite{Ly24}, supervised by the first named author, Liu. As a continuation of the work in \cite{DSX23}, Deng et al. \cite{DLSX26} derived an estimate for the total variation distance and further extended it to the Wasserstein-$p$ distance for $0<p<1$.

\medskip

However, a quantitative result on the pointwise convergence of the densities has remained open, even under smooth coefficient assumptions. In this work, we fill this gap by investigating the convergence of heat kernels from non-Gaussian diffusion (see SDE \eqref{e1}) to Gaussian diffusion (see SDE \eqref{e2}), where the drifts are of low regularity and possibly unbounded, satisfying condition ($\mathbf{H}_b^\beta$). Our first main contribution, Theorem \ref{th-1.1}(ii), is to establish that, even in this more challenging setting, the convergence rate remains $2-\alpha$. The proof of this result relies crucially on the uniform heat kernel estimates with respect to $\alpha$ established in Theorem \ref{th-1.1}(i). This uniform estimate is, to the best of our knowledge, new in the literature and is of independent interest.

\subsection{$\alpha$-dependence of invariant measures}
Having established the continuous dependence of the heat kernels for solutions to SDEs on the index $\alpha$ (see Theorem \ref{th-1.1}), we then apply this result to derive, under an additional dissipation condition, the (weighted) total variation and Kantorovich distances for the associated invariant measures (see Theorem \ref{th-6.1}).

\medskip

There are a few results concerning the convergence of invariant measures for stochastic systems in the transition from non-Gaussian to Gaussian noise. In this direction, we refer the reader to a series of works \cite{Liu22b, Liu22c, LL25}, in which the authors investigated the convergence of ergodic measures for three types of stochastic equations driven by cylindrical $\alpha$-stable processes as $\alpha \to 2$: the stochastic real Ginzburg-Landau equation (cf. \cite{Liu22b}), the stochastic Burgers equation (cf. \cite{Liu22c}), and the stochastic Navier-Stokes equation (cf. \cite{LL25}). These works provide a qualitative understanding of the convergence.

\medskip

For quantitative results, we recall~\cite{DSX23, Ly24, DLSX26}, where the convergence rate $2-\alpha$ for invariant measures was obtained by establishing distance estimates between the laws of solutions to SDEs. Due to the use of Malliavin calculus, the conditions in~\cite{DSX23, Ly24, DLSX26} impose high regularity requirements on the drift coefficient $b$; more precisely, $b$ is required to be third-order differentiable. In comparison, by employing heat kernel estimates, we prove the optimal quantitative convergence rate in the weighted total variation distance under our much weaker condition ($\mathbf{H}_b^\beta$) on the drift $b$, which is both low in regularity and allowed to be unbounded.

\begin{table}[H]
\caption{Results on the $\alpha$-dependence of ergodic measures for SDEs}
\vspace{1em}
\centering
\begin{tabular}{c|c|c|c}
\hline
\bf References & \bf Convergence & \bf Method &\bf  Conditions \\
\hline
\cite{DSX23} \&~\cite{Ly24} & Wasserstein-1 ($W_1$) & Malliavin calculus & \makecell{$b\in C^3$ \\ \it  (high regularity)} \\
\hline
\cite{DLSX26} & \makecell{Total variation (TV),\\$W_p$ ($0<p<1$)} & \makecell{Malliavin calculus} & \makecell{$b\in C^3$ \\ \it   (high regularity)} \\
\hline
\it Present & \makecell{ TV, $W_p$ ($0<p\leq 1$), \\ Weighted TV} & \makecell{Heat kernel estimates,\\ Harris's theorem} & \makecell{$(\mathbf{H}_b^\beta)$ \\ \it  (much weaker)} \\
\hline
\end{tabular}
\end{table}

\medskip

Furthermore, noting that the usual Wasserstein-$p$ distance and the total variation distance are both dominated by the weighted total variation distance with weight function $V_p(x)=1+|x|^p$, one sees that in addition to the methodological innovations, a further contribution of this work lies in the generalization of existing conditions and the establishment of the optimal rate for the weighted total variation distance. The use of heat kernel estimates leads to a significant relaxation of the regularity assumptions on the drift coefficient $b$. While semigroup estimates have been commonly used in studying the $\alpha$-dependence, the potential of pointwise heat kernel estimates in this context seems to have been explored to a lesser extent.

%%%%%%%%%%%%%%%%
%Conventions and notations
%%%%%%%%%%%%%%%%

\subsection{Conventions and notations}
Throughout this paper, we use the following conventions and notations: As usual, we use $:=$ as a way of definition. Define $\mN_0:= \mN \cup \{0\}$ and $\mR_+:=[0,\infty)$. The letter $c$ without subscripts will denote an unimportant positive constant, whose value may change in
different occasions. We use $A \asymp B$ and $A\lesssim B$ to denote $c^{-1} B \leq A \leq c B$ and $A \leq cB$, respectively, for some unimportant constant $c \geq 1$. We also use $A   \lesssim_c  B$ to denote $A \leq c B$ when we want to emphasize the constant.
\begin{itemize}
\item  Denote the identity $d\times d$-matrix by $\mI_{d \times d}$. $\sigma^*$ is the transformation of the matrix $\sigma$.
\item For $j \in \mathbb{N}$, $\nabla^j$ denotes the $j$-order gradient. $\nabla^0$ denotes the identity operator.
\item The supremum norm $\|\cdot\|_\infty$ is defined as $\|f\|_\infty:=\sup_{x\in\mR^d}|f(x)|.$
\item Denote by Beta functions and Gamma functions, respectively,
\begin{align*}
{\rm B}(s_1,s_2):= \int_0^1 x^{s_1-1}(1-x)^{s_2-1} \dif x, \ \ \forall s_1,s_2>0
\end{align*}
and
\begin{align}\label{eq:Gamma}
\Gamma(s) := \int_0^\infty x^{s-1} \e^{-x} \dif x,\ \ \forall s>0.
\end{align}
Furthermore, notice that
\begin{align}\label{eq:Beta}
\int_0^t (t-s)^{\gamma_1-1} s^{\gamma_2-1} \dif s = t^{\gamma_1+\gamma_2-1} {\rm B}(\gamma_1,\gamma_2)
\end{align}
and
\begin{align}\label{eq:ZG10}
{\rm B}(s_1,s_2)=\frac{\Gamma(s_1) \Gamma(s_2) }{\Gamma(s_1+s_2) }.
\end{align}
\item Let $\omega_{d-1}:=\frac{2\pi^{d/2}}{\Gamma(d/2)}$ be the surface measure of the unit sphere of $\mathbb{R}^d$.
\item For $\alpha \in (0,2)$, define
\begin{align}\label{eq:ZG03} 
\cC(d,\alpha):= \frac{\alpha
\Gamma(\frac{d+\alpha}{2})
}{2^{1-\alpha}\pi^{\frac{d}{2}}\Gamma(\frac{2-\alpha}{2}) }.
\end{align}
\item In the sequel, we use the conventions $0^0=1$ and $\frac{1}{0} =\infty$.
\item For functions $f,g$, define
\begin{align}\label{eq:ZG04}
 (f  \odot  g)_r(s,t;x,y) :=\int_{\mR^d} f (s,r;x,z)  g (r,t;z,y) \dif z
\end{align}
and
\begin{align}\label{eq:ZG06}
  (f  \otimes  g) (s,t;x,y):=\int_s^t (f  \odot  g)_r (s,t;x,y) \dif r.
\end{align}
\end{itemize}

\subsection{Outline of paper}

The remainder of this paper is organized as follows. In Section \ref{sec:2}, we present the main results of this work. Section \ref{sec:3} introduces some basic concepts and properties of isotropic $\alpha$-stable processes, along with elementary estimates. In Section \ref{sec:heat-s}, we establish uniform estimates and $\alpha$-continuity results for heat kernels, which serve as the foundation for the subsequent analysis. In Section \ref{sec:th1}, we study the densities of solutions to SDEs and prove Theorem \ref{th-1.1} using the parametrix method, which relies on the estimates from Sections \ref{sec:3} and \ref{sec:heat-s}. Finally, building on Theorem \ref{th-1.1}, we prove the optimal convergence rates for invariant measures, namely Theorem \ref{th-6.1}, in Section \ref{sec:last}.

\section{Main results}\label{sec:2}
\subsection{Heat kernel estimates}

In this paper, we always suppose that the drift coeffient $b$  satisfies that
\begin{itemize}
\item[($\mathbf{H}_b^\beta$)] for some $\beta\in (0,1]$ and $\kappa_0>0$,
$$
|b(x)-b(y)|\leq \kappa_0(|x-y|^\beta \vee|x-y|).
$$
\end{itemize}
Under ($\textbf{H}_b^\beta$), it is well known that for each starting point $x \in \mathbb{R}^d$, the SDE \eqref{e1} admits a unique weak solution, and that for any $0 \le s < t < \infty$, $X^{(\alpha)}_{s,t}(x)$ possesses a density $\hp^{(\alpha)}_{b}(s,t;x,\cdot)$ (see~\cite{MZ22}). An analogous result for the SDE \eqref{e2} can be found in~\cite{MPZ21}.

\medskip

To show our results on heat kernels, we first introduce the deterministic regularized flow associated with the drift $b$, which is found in~\cite{MPZ21,MZ22}.
Letting $\rho$ be a nonnegative smooth function with support in the unit ball of $\mR^d$ such that $\int_{\mR^d} \rho(x) \dif x =1$, and $\rho_\epsilon(\cdot):=\epsilon^{-d} \rho(\epsilon^{-1} \cdot)$ with $\epsilon>0$. That is, $\{\rho_\epsilon\}_{\epsilon>0}$ represents a family of standard mollifiers. For
fixed $\epsilon>0$ and $(s,x)\in \mR_+ \times \mR^d$, the following regularized ODE admits
a unique solution  $\theta^{(\epsilon)}_{s,t}(x)$:
\begin{equation}\label{eq:flow}
\dot{\theta}_{s,t}^{(\epsilon)}(x)=b_\epsilon(\theta^{(\epsilon)}_{s,t}(x)), ~t \geq 0,
\quad
\theta^{(\epsilon)}_{s,s}(x)=x.
\end{equation}
where
$$
b_\epsilon(x):=(b*\rho_\epsilon)(x):=\int_{\mR^d} b(y)\rho_\epsilon(x-y) \dif y,
$$
Moreover, from Lemma 1.1 of~\cite{MPZ21}, and Lemma 2.1 of~\cite{MZ22}, we have the following result.

\begin{lemma}[Deterministic flow]\label{lem:flow}
Under $(\mathbf{H}_b^\beta)$, for each $\epsilon>0$ and $s,t \geq 0$, the mapping $x\mapsto \theta^{(\epsilon)}_{s,t}(x)$ is a $C^1$-diffeomorphism and its inverse is given by $x\mapsto \theta^{(\epsilon)}_{s,t}(x)$. Moreover, for all $s,r,t\geq 0$,
\begin{align*}
\theta^{(1)}_{s,t}(x) = \theta^{(1)}_{r,t}\circ \theta^{(1)}_{s,r}(x).
\end{align*}
For any $T>0$, there exists a constant $c=c(d,T,\kappa_0)\geq 1$ such that for all $s,t \in [0, T]$ and  $x,y\in \R^d$,
\iffalse
$$
|\theta^{(1)}_{s,t}(x)-y|+|t-s|\asymp_c
|\theta^{(\epsilon)}_{s,t}(x)-y|+|t-s| \asymp_c
|x-\theta^{(\epsilon)}_{t,s}(y)|+|t-s|
$$
and
\fi
\begin{align*}
|\theta^{(1)}_{s,t}(x)-y| \asymp_c |x-\theta^{(1)}_{t,s}(y)|, \quad |\theta^{(1)}_{s,t}(x)-\theta^{(1)}_{s,t}(y)| \asymp_c |x-y|.
\end{align*}
\end{lemma}

\br
Following Remark 1.3 of~\cite{MZ22}, for the case $\alpha > 1$ considered in the present work, we simply set $\epsilon = 1$ instead of $\epsilon = (t-s)^{1/\alpha}$.
\er

Now we present our first main result. For the simplicity of notations, we define the following parameter set
$$
\Theta:= (d,\kappa_0,\beta),
$$
and introduce 
\begin{align*}
\varrho^{(k)}_{\gamma_1,\gamma_2}  (t,x):= t^{-(d+k)/\gamma_1} \wedge \frac{t}{|x|^{d+k+\gamma_2}},\ \ \text{for}\ \ \gamma_1,\gamma_2\in (0,2], k \geq 0,
\end{align*}
and
$$
\phi^{(\eta)}_{\gamma_1,\gamma_2} (s,t;x,y) :=
(t-s)^{\eta-1}\varrho^{(0)}_{\gamma_1,\gamma_2} (t-s, \theta^{(1)}_{s,t}(x)-y),
$$
where $  \theta^{(1)}_{s,t}$ is defined by \eqref{eq:flow}. In particular, when $\gamma_1=\gamma_2=\alpha$, let
$$
\varrho^{(k)}_\alpha (t,x):=\varrho^{(k)}_{\alpha,\alpha} (t,x),\quad \phi^{(\eta)}_{\alpha} (s,t;x,y):= \phi^{(\eta)}_{\alpha,\alpha} (s,t;x,y);
$$
if $k=0$,  let
\begin{align*}
\varrho^{(0)}_{\gamma_1,\gamma_2}  (t,x) := \varrho_{\gamma_1,\gamma_2}  (t,x).
\end{align*}
For $T \in (0,+\infty)$, denote
$$
\D_T:=\{ (s,t;x,y) \mid 0\leq s< t\leq T, x, y\in \mR^d\}.
$$

\bt[Heat kernel estimates] \label{th-1.1}
Fix $T>0$ and suppose that the condition ($\mathbf{H}_b^\beta$) holds. Then

\medskip\noindent
{\rm (i) (uniform estimates)}
for any  $\alpha \in [\alpha_0,2]$ with some $\alpha_0\in (1,2)$, there is a constant $c=c(\alpha_0,T,\Theta)>0$ such that for every $(s,t;x,y) \in \D_T$,
\begin{equation}\label{eq:main01}
|\hp_{b}^{(\alpha)}(s,t;x,y)|
\leq c   \varrho_\alpha
(t-s,\theta^{(1)}_{s,t}(x)-y);
\end{equation}

\medskip\noindent
{\rm (ii) ($\alpha$-continuity)}  for any  $\alpha \in [\alpha_0,2]$ with some $\alpha_0\in (\frac{12}{7},2)$, there exists  a constant $c=c(\alpha_0,T,\Theta)>0$ such that for each  $(s,t;x,y) \in \D_{T}$,
\begin{equation}\label{eq:main02}
|\hp_{b}^{(\alpha)}-\hp_{b}^{(2)}|(s,t;x,y)
\leq   c  (2-\alpha) \sum_{\gamma_1,\gamma_2\in\{\alpha,2\}}\phi^{(\frac{
7\alpha-12}{2\alpha})}_{\gamma_1,\gamma_2}(s,t;x,y).
\end{equation}
\et

As a special case (see Theorem~\ref{thm:WE01} and Theorem~\ref{thm:AA01}), we have the following similar estimates for the transition densities, denoted by $p^{(\alpha)}$, of $\alpha$-stable processes themselves.

\bp\label{prop:1.6} 
(i) Assume that $\alpha \in [\alpha_0,2]$ with
some $\alpha_0 \in (0,2)$, and $j \in \mN_0$. Then there is a constant $c >0$ depending only on
$d, \alpha_0,j$ such that for any $t> 0$,
$$
|\nabla^j p^{(\alpha)}(t,x)| \leq c \varrho^{(j)}_\alpha(t,x).
$$
(ii) Fix $T>0$. Assume that $\alpha \in [\alpha_0 ,2]$ with some $ \alpha_0 \in (0,2)$, and $j =1,2$. Then there is a constant $c >0$ depending only on $ d,T, \alpha_0 ,j$ such that for any $  t\in( 0,T]$,
\begin{align}\label{eq:HU89}
\begin{split}  | p^{(\alpha)}(t,x)- p^{(2)}(t,x)|
\leq c  (2-\alpha) (1+|\ln t|) (1+
t^\frac{\alpha-2}{\alpha} )
\sum_{\gamma_1,\gamma_2\in\{\alpha,2\}}  \varrho_{\gamma_1,\gamma_2}
(t,x),
 \end{split}
\end{align}
and
\begin{align*}
\begin{split} 
& 
\qquad  | \nabla^j p^{(\alpha)}(t,x)- \nabla^j p^{(2)}(t,x)|\\
& 
\leq c (2-\alpha) (1+|\ln t|)(1+
t^\frac{\alpha-2}{\alpha} )\sum_{k=1}^j
\sum_{\gamma_1,\gamma_2\in\{\alpha,2\}}|x|^{2k -j} \varrho^{(2k)}_{\gamma_1,\gamma_2}
(t,x).
 \end{split}
\end{align*}
\ep

\br
The above results (i), and \eqref{eq:HU89} in (ii), remain valid for the stochastic integral $\int_0^t \sigma(r) dL_r^{(\alpha)}$ when the coefficient $\sigma(r)$ is invertible and bounded (see Theorem \ref{thm:WE01} and Theorem \ref{thm:AA01}(i)). For $j=1,2$, result (ii) holds under the additional condition that $\sigma(r)$ is diagonal (see Theorem \ref{thm:AA01}(ii)). This naturally raises the question of whether the diagonal assumption can be relaxed.
\er

We close this subsection with an example to illustrate the main result.

\bexa[On the distance between laws of solutions to SDEs]
For any  $\alpha \in [\alpha_0,2]$ with some $\alpha_0\in (\frac{12}{7},2)$, there exists  a constant $c=c(\alpha_0,T,\Theta)>0$ such that for each  $t \in (0,T]$ and $x,y \in \mR^d$, 
\begin{align*}
\left\| {\rm Law} (X_t^{(\alpha)}(x)) - {\rm Law} (X_t^{(2)}(y)) \right\|_{\var} \leq c \left(
 t^{-1/2}|x-y|+(2-\alpha)  \sum_{\gamma_1,\gamma_2\in\{\alpha,2\}} t^{\tfrac{
7\alpha-12}{2\alpha}-\frac{\gamma_2(d+\gamma_1)}{\gamma_1(d+\gamma_2)}}\right),
\end{align*}
where $\|\mu_1 - \mu_2\|_{\var}:=\sup_{\|h\|_\infty \leq 1}|\mu_1(h)-\mu_2(h)|$ denotes the usual total variance. Precisely, by Theorem \ref{th-1.1}  and the equation (1.16) of~\cite{MPZ21}, it is easy to check that  
\begin{align*}
&\| {\rm Law} (X_t^{(\alpha)}(x)) - {\rm Law} (X_t^{(2)}(y)) \|_{\var} \nonumber\\
&\qquad\leq  \int_{\mR^d} |\hp_{b}^{(2)}(0,t;x,z)-\hp_{b}^{(2)}(0,t;y,z) | \dif z
 +  \int_{\mR^d} |\hp_{b}^{(\alpha)}-\hp_{b}^{(2)}|(0,t;x,z) \dif z\\
 & \qquad \lesssim t^{-1/2}|x-y|+(2-\alpha) t^{\tfrac{
7\alpha-12}{2\alpha}-1} \sum_{\gamma_1,\gamma_2\in\{\alpha,2\}} \int_{\mR^d}
\varrho_{\gamma_1,\gamma_2}
(t, \theta^{(1)}_{0,t}(x)-y) \dif y\\
&\qquad \lesssim 
  t^{-1/2}|x-y|+(2-\alpha) \sum_{\gamma_1,\gamma_2\in\{\alpha,2\}} t^{\tfrac{
7\alpha-12}{2\alpha}-\frac{\gamma_2(d+\gamma_1)}{\gamma_1(d+\gamma_2)}},
\end{align*}
where we used \eqref{eq:LK01} in the last inequality.
\eexa

\subsection{$\alpha$-dependence of invariant measures}

To study invariant measures, we make an additional dissipative assumption on the drift coefficients:

\begin{itemize}
\item[$(\mathbf{H}_b^{\rm diss})$] For some $r
\geq 0$,  there are two constants $c_0,c_1>0$ such that
$$
\langle x,b(x) \rangle \leq -c_0|x|^{2+r} +c_1.
$$
\end{itemize}

It is well-known (see Theorem 1.2 of \cite{ZZ23}) that under even weaker assumptions (i.e. $r>-\alpha$),  the semigroup $(P^{(\alpha)}_t)_{t \geq 0}$ associated with the Markov transition process $X^{(\alpha)}$ defined by
$$
P^{(\alpha)}_t f(x):=\EX f (X^{(\alpha)}_t(x)),\ \  f \in
{C}_b(\mathbb{R}^d),
$$
is strong Feller, and has a unique invariant measure (or stationary
distribution) $\mu^{(\alpha)}$, that is ${P_t^{(\alpha)}}^*
\mu^{(\alpha)}=\mu^{(\alpha)}$, for all $t\geq 0$, in the sense of
$$
\< {P_t^{(\alpha)}}^*\mu^{(\alpha)},f \> =\< \mu^{(\alpha)}, P^{(\alpha)}_t f \> =   \<\mu^{(\alpha)},f \>, \ \ f \in \mathscr{B}_b(\mR^d).
$$
Actually, for a probability measure $\mu$, ${P_t^{(\alpha)}}^*\mu$ is a
measure given by
$$
\int f(x) {P_t^{(\alpha)}}^* \mu(\dif  x) =\int P^{(\alpha)}_t f(x)   \mu(\dif  x),\ \ f \in \mathscr{B}_b(\mR^d).
$$

Before giving our second main result, we introduce some definitions about the distance between two probability measures. Let $V : \R^d \rightarrow [1, \infty )$ be a measurable function, and $\mu_1,\mu_2$ two probability measures. Denote the weighted total variance distance between $\mu_1$ and $\mu_2$ by
$$
\| \mu_1-\mu_2 \|_{\var,V}:=\sup_{\|h/V\|_\infty \leq 1}|\mu_1(h)-\mu_2(h)|.
$$
For the simplicity of notations, we define the following parameter set
$$
\widetilde \Theta:= (d,\kappa_0,\beta,c_0,c_1,r).
$$
Now we state our second main theorem.

\begin{theorem}[Optimal convergence rate]\label{th-6.1}
Assume that conditions ($\mathbf{H}_b^\beta$) and $(\mathbf{H}_b^{\rm diss})$ hold. Then for every  $\alpha \in [\alpha_0,2]$ with $\alpha_0 \in  (\frac{12}{7},2)$, there is a constant $
c=c(\alpha_0, p, \widetilde \Theta)>0$ such that
$$
\|\mu^{(\alpha)}-\mu^{(2)}\|_{\var,V_p}\leq c(2-\alpha),
$$
where $V_p(x)=1+|x|^p$ with $p \in (0,\alpha)$.
\end{theorem}

\begin{remark}
(i) If  $V_p=1$, then the theorem above arrives at the total variance distance.

\noindent
(ii)
The optimal transportation
cost between two probability measures is defined as: for $p\geq 0$,
$$
\mathcal{T}_p(\mu_1,\mu_2):= \inf\left\{  \EX |X-Y|^p, ~ \mathrm{Law} (X)=\mu_1,~ \mathrm{Law} (Y)=\mu_2\right\}.
$$
Due to Proposition 7.10 in \cite{Villani-2}, the theorem above also holds for $\mathcal{T}_p$ with $p\in(0,\alpha)$. Notice that by Theorem 7.3 of \cite{Villani-2}, one sees that when $p\in [1,+\infty)$, $\mathcal{T}_p = \dis_{W_p}^p$; $p \in [0,1)$, $\mathcal{T}_p= \dis_{W_p}$, where $\dis_{W_p}$ is the Kantorovich or Wasserstein-$p$ metric.
\end{remark}

\begin{remark}
Due to Proposition 5.1 of~\cite{DSX23} and Proposition 5.1 of~\cite{DLSX26}, we know that  $2-\alpha$ is the optimal rate in the sense of Wasserstein-1 metric and the total variation distance. In the fact, the same example demonstrates that $2-\alpha$ is also the optimal rate for  the weighted total variance distance. Precisely, consider the Ornstein-Uhlenbeck case, i.e. $b(x)=-x$.  Then the ergodic measure  $\mu^{(\alpha)}$ is given by the law of $\alpha^{-1/\alpha}L^{(\alpha)}_1$.   Using the method of characteristic function, one sees that
\begin{align*}
 ||\mu^{(\alpha)} -\mu^{(2)} ||_{\var,V_p}\geq ||\mu^{(\alpha)} -\mu^{(2)} ||_{\var} \geq \left| \EX \cos(\alpha^{-1/\alpha}L^{(\alpha)}_1)- \EX \cos ( \tfrac{1}{\sqrt{2}}L^{(2)}_1) \right| \geq c(2-\alpha).
\end{align*}
\end{remark}

\section{Preliminaries}\label{sec:3}

\subsection{Isotropic $\alpha$-stable processes with $\alpha \in (0,2]$}

In this subsection, we introduce some basic concepts and properties of isotropic $\alpha$-stable processes.  Let $(L_t^{(\alpha)})_{t \geq 0}$ be a $d$-dimensional rotationally invariant $\alpha$-stable process with $\alpha \in (0,2]$, that is its characteristic exponent is given by
\eqref{eq:ZG02}.

Since the $\alpha$-stable  process $L^{(\alpha)}$ with $\alpha \in(0,2]$ has the scaling property, $(\lambda^{-1/\alpha} L_{\lambda t}^{(\alpha)})_{t\geq 0} \stackrel{(d)}{=} ( L_t ^{(\alpha)})_{t \geq 0}$ for $\forall \lambda >0$, it is easy to see that for every $\alpha \in(0,2]$,
$$
p^{(\alpha)}(t,x) = \lambda^{d/\alpha} p^{(\alpha)}(\lambda t,\lambda ^{1/\alpha} x),\ \ \forall t > 0,
$$
where $p^{(\alpha)}(t,x)$ is defined by \eqref{eq:hp01} and is the density of $L_t^{(\alpha)}$.
In particular,
\begin{align}\label{eq:SS01}
p^{(\alpha)}(t,x) = t^{-d/\alpha}p^{(\alpha)}(1,t^{-1/\alpha}x),\ \ \forall t > 0.
\end{align}
Moreover, by~\cite[Eq.(19)]{BGR14}, we have that for each $\alpha\in (0,2)$,
\begin{align}\label{eq:KL01}
p^{(\alpha)}(t,x)\leq \frac{d4^\alpha \Gamma(\frac{d+\alpha}{2})}{2(1-2^{-d}) \pi^{d/2} (1-\e^{-1})} \frac{t}{|x|^{d+\alpha}}.
\end{align}
As for $\alpha=2$, the situation becomes into the Gaussian density, that is
\begin{align}\label{eq:PP04}
p^{(2)}(t,x) =(4\pi t)^{-\frac{d}{2}}\e^{-\frac{|x|^2}{4t}},\ \  t>0, x \in \mR^d,
\end{align}
which satisfies the following fact:
\begin{align}\label{eq:IU01-XM}
p^{(2)}(t,x+y) \leq 2^{d/2} p^{(2)}(2t,x) \e^{\frac{|y|^2}{4t}}.\end{align}

\subsubsection{The subordination representation}

For $\alpha\in(0,2)$, it is well known that the~$d$-dimensional isotropic $\alpha$-stable process can be represented by
\begin{align}\label{25eq:AA00}
L_t^{(\alpha)} = W_{S^{(\alpha)}_t},
\end{align}
where $\{W_t\}_{t\geq 0}$ is a $d$-dimensional  Brownian motion, and $\{S^{(\alpha)}_t\}_{ t\geq 0}$ is an ${\alpha} / {2}$-stable subordinator which is independent of Brownian
motion $\{W_t\}_{t\geq 0}$ and  is a nonnegative one-dimensional
L\'{e}vy process having the following Laplace transform:
$$
\EX \e^{-\lambda S^{(\alpha)}_t}
=\e^{-t \lambda^\frac{\alpha}{2}},\ \  \text{for}\ \ \lambda>0, t \geq 0.
$$
Hence, an isotropic $\alpha$-stable process with $\alpha
\in (0,2)$ is also called a {\it subordinated Brownian motion}.
From the Laplace transform, it is easy to see that $\lambda^{-2/\alpha} S_{\lambda t}^{(\alpha)} \overset{(d)}{=} S_t^{(\alpha)}$, for $ \lambda >0$, and then the probability density of the subordinator $S_t^{(\alpha)}$ has the following scaling property:
\begin{align}\label{eq:WW01}
\eta_\alpha(t,y)=\lambda^{2/\alpha}\eta_\alpha(\lambda t ,\lambda^{2/\alpha} y).
\end{align}
Consequently, by the subordination representation \eqref{25eq:AA00}, we get that
\begin{align}\label{eq:KL03}
p^{(\alpha)}(t,x)=\EX p^{(2)}(S^{(\alpha)}_t,x)=\int_0^\infty p^{(2)} (u,x)\eta_\alpha(t,u)\dif u, \ \ \alpha \in (0,2),
\end{align}
where $p^{(2)}(t,x)$ is the Gaussian probability density given by \eqref{eq:PP04}. By Lemma 4.1 of~\cite{DS19}, or Eq.(25.5) in p.162 of~\cite{Sa99}, the moments of order $\vartheta \in(-\infty,\alpha/2)$ of an $\alpha/2$-stable subordinator $(S^{(\alpha)}_t)_{t\geq 0}$ are given by
\begin{align}\label{eq:SS}
\mE (S_t^{(\alpha)})^\vartheta = \frac{\Gamma(1- 2\vartheta/\alpha )}{\Gamma(1-\vartheta)} t^{2\vartheta/\alpha},\ \  t >0.
\end{align}

\subsubsection{Infinitesimal  generators and L\'evy measures}

It is well-known that the infinitesimal  generator of a Brownian motion is the Laplacian $  \Delta:=
\sum_{i=1}^d \p_{x_i}^2$,  and  the one of $\alpha$-stable process is the fractional Laplacian operator $\Delta^{\alpha/2}:=-(-\Delta)^{\alpha/2}$ defined by
\begin{align*}
\begin{split}
\Delta^{\alpha/2} \varphi (x)
& :=  \int_{\mR^d} \( \varphi(x+z) - \varphi(x) - z \1_{|z|\leq 1} \cdot \nabla \varphi(x) \)\nu^{(\alpha)}(\dif z)\\
& = \frac{1}{2} \int_{\mR^d} \( \varphi(x+z) + \varphi(x-z) - 2\varphi(x) \)\nu^{(\alpha)}(\dif z),
\end{split}
\end{align*}
where $\nu^{(\alpha)}(\dif z) := \cC(d,\alpha)|z|^{-d-\alpha}$ is called the {\it L\'evy measure}, and the constant $\cC(d,\alpha)$  is defined by \eqref{eq:ZG03}. Notice that operators $\Delta^{\alpha/2}$ are symmetric in the sense of
\begin{align}\label{eq:RR02}
\<f, \Delta^{\alpha/2} g\>=\<\Delta^{\alpha/2} f,  g\>,\ \ \alpha \in (0,2].
\end{align}
Here and below we use the convention $\Delta^{2/2} = \Delta$. In fact,  the density $p^{(\alpha)}(t,x)$ introduced before is the heat kernel of the operator $\Delta^{\alpha/2}$ for each $\alpha \in (0,2]$, i.e.,
\begin{align}\label{eq:RR01}
\p_t p^{(\alpha)}(t,x) = \Delta^{\alpha/2} p^{(\alpha)}(t,x),\ \ \lim_{t \downarrow 0} p^{(\alpha)}(t,x) = \delta_0(x),
\end{align}
where $\delta_0$ is the Dirac measure at point zero.

We conclude this subsubsection by presenting some trivial but useful facts about L\'evy measure $\nu^{(\alpha)}$ for $\alpha \in (0,2)$. By Lemma 2.2 of ~\cite{DSX23}, there is a constant $c>0$ independent of $\alpha \in (0,2)$ such that
\begin{align}\label{eq:JH00}
\omega_{d-1} \cC(d,\alpha) \leq c d( 2-\alpha)
\end{align}
and
\begin{align}\label{C-d-alpha} 
\left|\frac{\omega_{d-1} \cC(d,\alpha)}{ d (2-\alpha)}-1\right| \leq c
(2-\alpha) \log(1+d),
\end{align}
where $\cC(d,\alpha)$ is defined by \eqref{eq:ZG03}. Hence, fixed $\delta>0$, for any $ \vartheta >\alpha $,
\begin{align}\label{eq:JH01}
\int_{ |z|\leq \delta}|z|^\vartheta  \nu^{(\alpha)}(\dif z)=\frac{\omega_{d-1} \cC(d,\alpha)}{\vartheta -\alpha} \delta^{\vartheta -\alpha} \lesssim_d  \frac{2-\alpha }{\vartheta -\alpha} \delta^{\vartheta -\alpha}, \end{align}
and for $\vartheta < \alpha$,
\begin{align}\label{eq:JH02}
\int_{ |z|>\delta}|z|^\vartheta  \nu^{(\alpha)}(\dif z)=\frac{\omega_{d-1} \cC(d,\alpha)}{\alpha-\vartheta} \delta^{\vartheta -\alpha}\lesssim_d \frac{2-\alpha}{\alpha-\vartheta} \delta^{\vartheta -\alpha}.
\end{align}
Furthermore, it is easy to see that, for any $\gamma_2>\alpha>\gamma_1\geq 0$,
\begin{align*}
 \int_{|z|\leq 1} |z|^{\gamma_2} \nu^{(\alpha)}(\dif z) + \int_{|z|>1} |z|^{\gamma_1}\nu^{(\alpha)}(\dif z)
& =\omega_{d-1} \cC(d,\alpha) \left( \tfrac{1}{\gamma_2 - \alpha} + \tfrac{1}{\alpha - \gamma_1} \right) \\
& \lesssim_d (2-\alpha) \left( \tfrac{1}{\gamma_2 - \alpha} + \tfrac{1}{\alpha - \gamma_1} \right)
<\infty,
\end{align*}
and for any $\gamma \in [0,2)$,
\begin{equation}\label{un-alpha-0}
 \sup_{ \alpha \in (0, 2)} \int_{ |z|\leq 1} |z|^{2}
  \nu^{(\alpha)}(\dif z)+\sup_{\alpha \in [\frac{\gamma+2}{2}, 2)}
  \int_{ |z|> 1} |z|^\gamma \nu^{(\alpha)}(\dif z)<\infty.
\end{equation}

%%%%%%%%%%%%%%%%%%%%%%%%%%%%%%%%%%%%
\subsection{Some basic inequalities of $\varrho^{(k)}_{\gamma_1,\gamma_2}(t,x)$}
%%%%%%%%%%%%%%%%%%%%%%%%%%%%%%%%%%%%%%%%%%%

For $k \in [0,\infty)$ and $\gamma_i \in (0,\infty)$,  $i=1,2$, define
\begin{align}\label{eq:ZG100}
\varrho_{\gamma_1,\gamma_2}^{(k)} (t,x):= t^{-(d+k)/\gamma_1} \wedge \frac{t}{|x|^{d+k+\gamma_2}}.
\end{align}
Simply,  when $\gamma_1=\gamma_2=\gamma$, we denote
\begin{align*}
\varrho^{(k)}_{\gamma}(t,x):=\varrho^{(k)}_{\gamma,\gamma}(t,x).
\end{align*}
In particular, we drop the superscript $(k)$ when $k=0$, and denote
\begin{align}\label{eq:ZG05}
\varrho_\gamma(t,x):=\varrho_{\gamma,\gamma}(t,x).
\end{align}
Trivially, it is easy to see that
\begin{align}\label{eq:WK01}
\varrho^{(k)}_{\gamma_1,\gamma_2}(t,x) =
\begin{cases}
\frac{t}{|x|^{d+k+\gamma_2}}, & \text{when}\ \  |x|\geq  t^{\frac{d+k+\gamma_1}{\gamma_1(d+k+\gamma_2)}} ,\\
 t^{-(d+k)/\gamma_1} , & \text{when}\ \ |x|\leq   t^{\frac{d+k+\gamma_1}{\gamma_1(d+k+\gamma_2)}} ,
\end{cases}
\end{align}
and then we have
\begin{align}\label{eq:HY00}
\varrho^{(k+j)}_{\gamma}(t,x) = (t^{-1/\gamma} \wedge |x|^{-1})^j \varrho^{(k)}_{\gamma}(t,x), \ \ \hbox{for}\ \  0 \leq j ,k,
\end{align}
and
\begin{align}\label{eq:FG01}
\frac{t}{\(t^{\frac{d+k +\gamma_1}{\gamma_1(d+k +\gamma_2)}} +|x|\)^{d+k +\gamma_2} }\leq \varrho^{(k )}_{\gamma_1,\gamma_2}(t,x) \leq 2^{d+k +\gamma_2} \frac{t}{\(t^{\frac{d+k +\gamma_1}{\gamma_1(d+k +\gamma_2)}} +|x|\)^{d+k+\gamma_2} }.
\end{align}
Observe that for every $c_1>0$, $c_2 >1$, and $\beta>0$, there is a constant $ c=\frac{c_2}{c_2-1}\vee c_1>0$,
\begin{align*}
(t^{1/\beta} + |x+z|)^{-1}\leq  c(t^{1/\beta} + |x|)^{-1},\ \ \hbox{when}\ \   |z| \leq (c_1 t^{1/\beta})\vee (|x|/c_2),
\end{align*}
which together with \eqref{eq:FG01} derives that
\begin{align}\label{eq:GH01}
\varrho^{(k)}_{\gamma_1,\gamma_2} (t,x+z) \leq (2c)^{(d+k+\gamma_2)} \varrho^{(k)}_{\gamma_1,\gamma_2} (t,x),\ \ \hbox{for}\ \  |z| \leq \left(c_1 t^{\frac{d+k+\gamma_1}{\gamma_1(d+k+\gamma_2)}} \right)\vee (|x|/c_2).
\end{align}

We prepare the following lemmas for later use.

\begin{lemma}\label{lem-2.4.2}
Let $k  \in \mN_0,\gamma_i \in [\alpha_0,2], i=1,2,3,4,$ with some $\alpha_0 \in (0,2)$.
 There is a constant $ c=c(d,\alpha_0,k)>0$, such that for all $x,y\in \mR^d$ and $t,s>0$,
\begin{align}\label{eq:HY01}
 \varrho_{\gamma_1,\gamma_2}^{(k)}(t, x-y)\varrho_{\gamma_3,\gamma_4}^{(k)}(s, y)  \leq c \left( \varrho^{(k)}_{\gamma_1,\gamma_2}(t, x-y)+\varrho^{(k)}_{\gamma_3,\gamma_4}(s, y) \right)  \sum_{\substack{i=1,3\\ j =2,4}} \varrho^{(k)}_{\gamma_i,\gamma_j  }(t+s,x).
\end{align}
\end{lemma}

\begin{proof}
Observe that
\begin{align*}
 \varrho^{(k)}_{\gamma_1,\gamma_2}(t,  x-y) & \wedge \varrho^{(k)}_{\gamma_3,\gamma_4}(s, y )
= t^{-\frac{d+k}{\gamma_1}}\wedge  \frac{t}{|x-y|^{d+k+\gamma_2}}\wedge s^{-\frac{d+k}{\gamma_3}} \wedge
\frac{s}{|y |^{d+k+\gamma_4}}  \\
&\leq \left[\left(\frac{t+s}{2}\right)^{-\frac{d+k}{\gamma_1}}\vee
\left(\frac{t+s}{2}\right)^{-\frac{d+k}{\gamma_3}}\right]\wedge \left[
 \frac{t+s}{(|x|/2)^{d+k+\gamma_2}} \vee  \frac{t+s}{(|x|/2)^{d+k+\gamma_4}} \right]\\
&\lesssim \sum_{\substack{i=1,3\\ j =2,4}} \varrho^{(k)}_{\gamma_i,\gamma_j  }(t+s,x),
\end{align*}
which derives the desired result since for any $a,b>0$,  $ab=(a\vee b)(a\wedge b) \leq (a+ b)(a\wedge b) $. The proof is complete.
\end{proof}

\br
When $\gamma_i =\alpha$, the inequality \eqref{eq:HY01} is known as  the (3P)-inequality (see Proposition 2.4 of~\cite{WZ15} for example): for all $t,s>0$ and $x,y,z \in \mR^d $,
$$
\varrho_\alpha(t,x-z) \varrho_\alpha(s,z-y) \leq c \( \varrho_\alpha(t,x-z) + \varrho_\alpha(s,z-y) \) \varrho_\alpha(s+t,x-y).
$$
\er

Moreover, by the polar formula, it is easy to check that, for $0\leq \vartheta< \gamma_2$,
\begin{align}\label{eq:LK01}
\int_{\mR^d}|x|^\vartheta \varrho_{\gamma_1,\gamma_2} (t,x) \dif x   \lesssim_d t^{1-\frac{(\gamma_2-\vartheta)(d+\gamma_1)}{\gamma_1(d+\gamma_2)}}.
\end{align}

\bl\label{lem:ZX01}
Assume that $j=0,1$ and $\alpha \in [\alpha_0,2)$ with some $\alpha_0\in(0,2)$. Then there is a constant $c=c(d,\alpha_0)>0$ such that for $s>0$ and $x\in \mR^d$,
\begin{align}\label{eq:ZZ09}
\int_{ |z|>s^{1/\alpha}\vee \frac{|x|}{2}}  \varrho_{\alpha} (s,x+z)\nu^{(\alpha)}(\dif z)
 \leq c
 (2-\alpha)  s^{-1}
 \varrho_{\alpha} (s,x) .
\end{align}
\el

\begin{proof}
It is similar to the proof of (2.28) in Theorem 2.4 from~\cite{CZ16}. If $|x|  \leq  s^{1/\alpha}$, then by the  definition  of $ \varrho_{\alpha} $ and \eqref{eq:JH02},
\begin{align*}
l.h.s ~\text{of} ~\eqref{eq:ZZ09}
 & \leq s^{-d/\alpha}  \int_{ |z|>s^\frac{1}{\alpha}} \nu^{(\alpha)}(\dif z)
 \lesssim  (2-\alpha) s^{-d/\alpha-1}     \\
 &
\overset{\eqref{eq:WK01}}{\lesssim}
 (2-\alpha) s^{-1}
\varrho_{\alpha} (s,x) .
\end{align*}
If $|x| >  s^{1/\alpha}$, then by  \eqref{eq:JH00}, we have that
\begin{align*}
l.h.s ~\text{of} ~\eqref{eq:ZZ09}
& \lesssim
(2-\alpha)  |x|^{-d-\alpha}   \int_{ |z|>  \frac{|x|}{2} }
 \varrho_{\alpha} (s,x+z)   \dif z \\
& \lesssim
 (2-\alpha)   |x|^{-d-\alpha}  \int_{\R^d}  \varrho_{\alpha} (s,x+z)   \dif z  \\
 &
  \overset{\eqref{eq:LK01}}{\lesssim} (2-\alpha)
|x|^{-d-\alpha}
 \overset{\eqref{eq:WK01}}{\lesssim}  (2-\alpha)  s^{-1} \varrho_{\alpha} (s,x)  .
\end{align*}
Combining the above calculus, we get the desired result.
\end{proof}

\bl
Assume that $k\in N_0$, and $T>0$, and $\gamma_1,\gamma_2 \in [\alpha_0,2]$ with some $\alpha_0 \in (1,2)$. Then for any $p \in [0,k]$, there is a constant $c=c(T,p,\alpha_0)>0$ such that for any $t\in [0,T]$ and $x \in \mR^d$,
\begin{equation}\label{eq:LRW01}
\frac{|x|^p\varrho_{\gamma_1,\gamma_2}^{(k)}(t,x) }{\varrho_{\gamma_1,\gamma_2}(t,x)} 
\leq c t^{p/\check \gamma-k/\hat\gamma}.
\end{equation}
where
$$
\check \gamma := \gamma_1 \vee \gamma_2, \qquad \hat \gamma := \gamma_1 \wedge \gamma_2.
$$
\el

\begin{proof}
If $ |x| \leq t^{\frac{d+\gamma_1}{\gamma_1(d+\gamma_2)}}$, then by \eqref{eq:WK01}, we have
$$
\frac{|x|^p\varrho_{\gamma_1,\gamma_2}^{(k)}(t,x) }{\varrho_{\gamma_1,\gamma_2}(t,x)}\leq \frac{|x|^p t^{-(d+k) /\gamma_1}}{t^{-d/\gamma_1} } \leq t^{\frac{p(d+\gamma_1)}{\gamma_1(d+\gamma_2)}-k/\gamma_1};
$$
If $ |x| \geq t^{\frac{d+\gamma_1}{\gamma_1(d+\gamma_2)}}$, then by \eqref{eq:WK01} again, we get
$$
\frac{|x|^p\varrho_{\gamma_1,\gamma_2}^{(k)}(t,x) }{\varrho_{\gamma_1,\gamma_2}(t,x)}\leq \frac{|x|^{p-d-k-\gamma_2}}{|x|^{-d-\gamma_2}}\leq |x|^{p-k}
=|x|^{\frac{p(d+\gamma_1)}{\gamma_1(d+\gamma_2)} - \frac{k(d+\gamma_1)}{\gamma_1(d+\gamma_2)}}.
$$
Hence,
\begin{align*}
\frac{|x|^p\varrho_{\gamma_1,\gamma_2}^{(k)}(t,x) }{\varrho_{\gamma_1,\gamma_2}(t,x)} \lesssim_{T,\alpha_0} t^{\frac{p(d+\gamma_1)}{\gamma_1(d+\gamma_2)} - \frac{k}{\gamma_1}\left ( \frac{d+\gamma_1}{d+\gamma_2} \vee 1 \right)}.
\end{align*}
The desired result follows from the fact that if $a,b,c,d \in \mR_+$ satisfy $\frac{a}{b}<\frac{c}{d}$, then $\frac{a}{b} < \frac{a+c}{b+d}< \frac{c}{d}$.
\end{proof}

\subsection{Some uniform estimates of densities}

First of all, recalling the elementary fact that for any real number $n >0$,
\begin{align*}
\e^{-u} \leq n^n  u^{-n}, \ \ \forall u \geq 0,
\end{align*}
one sees that for any real number $n \geq 0$ and constant $c>0$,
\begin{align}\label{eq:PP02}
\e^{-\frac{|x|^2}{c}} \leq (c n)^n |x|^{-2n},
\end{align}
where we used the conventions $0^0 =1$  and $\frac{1}{0}=\infty$.

The following lemma is similar to~Lemma 2.2 in \cite{CZ16} with a slight extension, and will be used in the proof of Theorem \ref{thm:WE01}.

\begin{lemma}\label{lem:ZT00}
Assume that $\alpha \in [\alpha_0,2]$ with some $\alpha_0\in (0,2)$. There is a constant $c >0$ depending only on $d, \alpha_0$ such that, for any $ t> 0$,
\begin{align}\label{eq:NH01-1}
p^{(\alpha)}(t,x) \leq c \varrho_\alpha(t,x),
\end{align}
where $p^{(\alpha)}$ is the density of $\alpha$-stable process given by \eqref{eq:hp01}, and $p^{(2)}$ and $\varrho_\alpha$ are defined by \eqref{eq:PP04} and \eqref{eq:ZG05} respectively.
\end{lemma}

\begin{proof}
{\bf (Step 1)} We first claim that there is a constant $c=c(d)>0$ such that
\begin{align}\label{eq:KL02}
p^{(2)}(t,x)\leq c \varrho_2(t,x).
\end{align}
Indeed,  from \eqref{eq:PP02}, picking $c =4t$, and taking $n=0$ and $n=(d+2)/2$ in turn, we get the desired estimate \eqref{eq:KL02} by \eqref{eq:PP04}.

\medskip\noindent
{\bf (Step 2)} By the subordination representation \eqref{eq:KL03}, it is easy to see that
\begin{align*}
p^{(\alpha)} (t,x)
\lesssim  \int_0^\infty
u^{-\frac{d}{2}}\eta_\alpha(t,u)\dif u = \mE [(S_t^{(\alpha)})^{-\frac{d}{2}}]
\overset{\eqref{eq:SS}}{ =} \frac{\Gamma(1+ d/\alpha)}{\Gamma(1+d/2)}  t ^{-d/\alpha},
\end{align*}
which,  combined with \eqref{eq:KL01} and \eqref{eq:KL02}, deduces that there is a constant $c=c(d,\alpha_0) >1$ such that for any $\alpha \in [\alpha_0,2]$,
\begin{align}p^{(\alpha)} (t,x)\leq c \left(t^{-d/\alpha} \wedge\frac{t }{|x|^{d+ \alpha}}\right)\overset{\eqref{eq:ZG05}}{=}c \varrho_\alpha(t,x),
\label{eq:BB02}
\end{align}
where we used the continuity of Gamma functions on $(0,\infty)$.
\end{proof}

\br
The proof above can be refined by using the fact that $p^{(\alpha)}(t,x)\leq p^{(\alpha)}(t,0)$. By \eqref{eq:hp01}, for each $t>0$, one sees that
\begin{align*}
p^{(\alpha)}(t,x)
 \leq (2\pi)^{-d} \int_{\mR^d} \e^{-t|\xi|^\alpha} \dif \xi
= p^{(\alpha)}(t,0)\overset{\eqref{eq:SS01}}{=} t^{-d/\alpha} p^{(\alpha)}(1,0),
\end{align*}
where
\begin{align*}
p^{(\alpha)}(1,0) &  = (2\pi)^{-d} \int_{\mR^d} \e^{-|\xi|^\alpha} \dif \xi
=  (2\pi)^{-d} \omega_{d-1}\int_0^\infty \e^{-r^\alpha} r^{d-1} \dif r\\
&  = (2\pi)^{-d} \frac{\omega_{d-1}}{\alpha} \int_0^\infty \e^{-u} u^{d/\alpha-1} \dif u\\
& \overset{\eqref{eq:Gamma}}{=} (2\pi)^{-d}\omega_{d-1}  \frac{\Gamma(d/\alpha)}{\alpha}.
\end{align*}
The item $ \frac{\Gamma(d/\alpha)}{\alpha} $ is controlled by some $c_{d,\alpha_0}$ since the continuity of Gamma functions on $(0,\infty)$.
\er

We also need the following lemma, which will be applied to prove Corollary \ref{cor:ER01}.

\bl\label{lem:EE02}
Assume that $T>0$ and $\alpha \in [\alpha_0,2]$ with some $\alpha_0\in (0,2)$. There is a constant $c >0$ depending only on $d,T, \alpha_0$ such that, for any $t \in( 0 ,T]$,
\begin{align*}
p^{(2)}(t,x) \leq c t^{\frac{\alpha-2}{2}}\varrho_\alpha(t,x).
\end{align*}
\el

\begin{proof} 
From \eqref{eq:PP02}, picking $c =4t$, and taking $n=0$ and $n=(d+\alpha)/2$ in turn, we get  that
 \begin{align*}
p^{(2)}(t,x)& = (4\pi t)^{-\frac{d}{2}}\e^{-\frac{|x|^2}{4t}}\leq (4\pi t)^{-\frac{d}{2}}  \( 1 \wedge
 \[ (4t \tfrac{d+\alpha}{2})^{\frac{d+\alpha}{2}}|x|^{-d-\alpha}\] \)
 \lesssim t^{-\frac{d}{2}} \wedge \frac{t^{\alpha/2}}{|x|^{d+\alpha}},
 \end{align*}
 which implies that
\begin{align*}
p^{(2)}(t,x) \lesssim (t^{\frac{d}{\alpha} - \frac{d}{2}}+t^{\frac{\alpha-2}{2}}) \varrho_\alpha(t,x)
 \lesssim t^{\frac{\alpha-2}{2}}\varrho_\alpha(t,x).
\end{align*}
The proof is finished.
\end{proof}

\section{Heat kernels: uniform estimates and $\alpha$-continuity}\label{sec:heat-s}

Let $\sigma: \mR_+ \to \mR^d \times \mR^d$ be a measurable $d \times d$-matrix-valued function satisfying the non-degeneracy condition: there is a constant $\kappa_1>1$ such that
\begin{align}\label{eq:XM01}
\kappa_1^{-1} |\xi|^2 \leq |\sigma (r) \xi|^2 \leq \kappa_1 |\xi|^2, \ \ \forall r\in \mR_+, \xi \in \mR^d.
\end{align}
In this section, we consider the following time-inhomogeneous Markov process:
\begin{align}\label{eq:XM00}
Z_{s,t}^{(\alpha),\sigma} = \int_s^t \sigma (r) \dif L_r^{(\alpha)},\end{align}
where $(L_t^{(\alpha)})_{t \geq 0}$  is an $\alpha$-stable process with $\alpha \in (0,2]$. It is well-known (see eg.~\cite{CZ18a}, or Lemma 2.8 of~\cite{MZ22}) that for any $0 \leq s< t < \infty$, the random variable $Z_{s,t}^{(\alpha),\sigma}$ has a smooth density $p^{(\alpha),\sigma} (s,t,x)$
satisfies the scaling property
\begin{align}
\label{eq:SS01-X}
p^{(\alpha),\sigma} (s,t,x) = (t-s)^{-d/\alpha} p^{(\alpha),\widetilde \sigma} (0,1,(t-s)^{-1/\alpha}x),
\end{align}
where
$$
\widetilde \sigma (r) := \sigma (s+r(t-s)).
$$
Notice that condition \eqref{eq:XM01} still holds for $\widetilde \sigma $. When $s=0$, we use a simpler notation $p^{(\alpha),\sigma} (t,x)$ if there is no risk of confusion. Moreover, the infinitesimal generators are given by
\begin{align}\label{eq:XM100}
& \quad \widetilde{\sL}_r^{(\alpha),\sigma} f(x) \nonumber\\
& :=\begin{cases}
\int_{\mR^d} \(f(x+ \sigma(r) z) - f(x) - \sigma(r)z\1_{|z|\leq 1}\cdot \nabla f(x)\) \nu^{(\alpha)} (\dif z), & \ \ \text{when}\ \ \alpha \in (0,2),\\
{\rm TR} \((\sigma \sigma^*)(r) {\rm H}(f)(x)\) ,& \ \ \text{when}\ \ \alpha=2,
\end{cases}
\end{align}
where $f \in C_b^2$, and    {\rm TR} denotes the trace of matrix, and {\rm H} is the Hessian matrix. It is well-known (see eg. Section 2 of~\cite{CZ18a}) that
\begin{align}\label{eq:RR02-X}
\p_s p^{(\alpha),\sigma}_{s,t} (x,y) + \widetilde{\sL}_{s}^{(\alpha),\sigma} p_{s,t}^{(\alpha),\sigma} (x,y) =0, ~   \p_t p^{(\alpha),\sigma}_{s,t} (x,y) - (\widetilde{\sL}_{t}^{(\alpha),\sigma})^* p_{s,t}^{(\alpha),\sigma} (x,y) =0, ~ \text{for } s<t,
\end{align}
with
\begin{align}\label{eq:RR03-X}
\lim_{s \uparrow t}p_{s,t}^{(\alpha),\sigma} (\cdot,y) = \delta_y (\cdot),\quad \lim_{t \downarrow s}p_{s,t}^{(\alpha),\sigma} (x,\cdot) = \delta_x (\cdot),
\end{align}
where the limit is taken in the weak sense, and $\delta_x$ is the Dirac measure at $x$.

\subsection{Uniform heat kernel estimates for non-local and local}\label{sec:XX00}

In this subsection, we establish density estimates for the random variable $Z_t^{(\alpha),\sigma}$ defined by \eqref{eq:XM00}, uniformly in $\alpha$. More precisely, we prove the following theorem, which corresponds to Lemmas 2.8 \& 2.10 of~\cite{MZ22}, and Lemma 2.1 \& Corollary 2.2 of~\cite{WH-SPA}.

\bt[Uniform estimates]
\label{thm:WE01}
Assume that $\alpha \in [\alpha_0,2]$ with some $\alpha_0\in (0,2)$, and $ j\in\mN_0$. For any $0\leq s < \infty$ and $t>0$, there is a constant $c >0$ depending only on $d,j, \alpha_0,\kappa_1$ such that
\begin{align}\label{eq:NH01}
 |\nabla^j p^{(\alpha),\sigma}(s,s+t,x)|\leq c ~     \varrho^{(j)}_\alpha(t,x),
\end{align}
and
\begin{align}\label{eq:NH02}
\begin{split}
& \quad |\nabla^j p^{(\alpha),\sigma}(s,s+t,x)-\nabla^j p^{(\alpha),\sigma}(s,s+t,\bar x)| \\
& \leq c
 ((t^{-1/\alpha} |x-\bar x|)\wedge 1) \(\varrho^{(j)}_\alpha(t,x) + \varrho^{(j)}_\alpha(t,\bar x)\);
\end{split}
\end{align}
and  for $j=0,1$ and $\alpha_0=j$,
\begin{align}\label{eq:NH03}
|\Delta^{\alpha/2} \nabla^j p^{(\alpha),\sigma}(s,s+t,x)| \leq c   t^{-1-j/\alpha} \varrho_{\alpha}(t,x),
\end{align}
where we use the convention $\Delta^{2/2} = \Delta$.
\et

\br
Notice that Eq. \eqref{eq:NH03} is included here only for completeness and is not used in this work.
\er

To prove the theorem, we first recall a representation of the heat kernel $p^{(\alpha),\sigma}$ from~\cite{MZ22}. Fix a c\`adl\`ag path $\ell_r$. Define
$$
A_t^{\ell} := \int_0^t (\sigma \sigma^*) (r) \dif \ell_r .
$$
From the non-degeneracy condition \eqref{eq:XM01}, one sees that the Gaussian random variable,
$$
W^{\sigma,\ell}_t : = \int_0^t \sigma (r) \dif W_{\ell_r},
$$
has a density given by
\begin{align}\label{eq:XM01-20241123}
g^{\sigma,\ell} (t,x) & : = \frac{ \exp \left\{-\frac{\< (A_t^{\ell})^{-1} x, x\>}{4}\right \}}{(4 \pi )^{d/2} \sqrt{\det (A_t^{\ell})}}. 
\end{align}
When $\ell_t=t$, we use $g^{\sigma} (t,x)$ denotes $g^{\sigma,\ell} (t,x)$ for simplicity. Noting that by \eqref{eq:XM01},
$$
\< (A_t^{\ell})^{-1} x, x\> \asymp \frac{|x|^2}{\ell_t},\ \ \text{and }\det ((A_t^{\ell})^{-1}) \asymp \ell_t^{-d},
$$
where the implicit constants only depend on  $d,\kappa_1$,  we have that
\begin{align}\label{eq:XM06}
g^{\sigma,\ell}  (t,x) \asymp (4 \pi \ell_t )^{-d/2} \exp \left\{-\frac{ |x|^2}{4\ell_t}\right \} = p^{(2)} (\ell_t,x),
\end{align}
where $p^{(2)}(t,x)$ is the Gaussian density defined by \eqref{eq:PP04}.

 Observe that, when $\ell_t = t$, $(W_{\ell_t})_{t \geq 0}$ is the Brownian motion, and then the density of $Z_t^{(2),\sigma}\overset{(d)}{=}  \int_0^t \sigma (r) \dif W_r $ defined by \eqref{eq:XM00} satisfies
\begin{align}\label{eq:XM03}
p^{(2),\sigma}(t,x)= g^{\sigma} (t,x)\overset{\eqref{eq:XM06}}{\asymp} p^{(2)}(t,x);
\end{align}
when $\alpha \in (0,2)$, taking $\ell$ be an $\alpha/2$-stable subordinator $S^{(\alpha)}$ independent of $W$, by the calculations above and the subordination representation \eqref{eq:KL03}, one sees that
the density of $Z_t^{(\alpha),\sigma} \overset{(d)}{=} W_t^{\sigma,S^{(\alpha)}}$ is given by
\begin{align}\label{eq:XM03-1}
p^{(\alpha),\sigma} (t,x) = \mE g^{\sigma,S^{(\alpha)}} (t,x).
\end{align}
Hence, by \eqref{eq:XM06},
$$
p^{(\alpha),\sigma} (t,x) \asymp \mE p^{(2)} (S_t^{(\alpha)},x) \overset{\eqref{eq:KL03}}{=} p^{(\alpha)} (t,x).
$$

\begin{lemma}\label{lem:EE01}
For $j \in \mN_0$, there is a constant $c=c(d,j,\kappa_1)>0$ such that
\begin{align}\label{eq:GG01}
|\nabla^j  p^{(2),\sigma}(t,x)|  \leq c   (t^{-1/2}\wedge |x|^{-1})^{j} p^{(2)}(2t,x),
\end{align}
and
\begin{align*}%\label{eq:GG02}
|\nabla^j p^{(2),\sigma} (t,x) - \nabla^j p^{(2),\sigma} (t,\bar x)| \leq c   t^{-j/2}   \( (t^{-1/2}|x-\bar x|)\wedge 1 \) \(p^{(2)}(4t,x) + p^{(2)}(4t,\bar x)\).
\end{align*}
\end{lemma}

\begin{proof}
By the scaling properties \eqref{eq:SS01} and \eqref{eq:SS01-X}, it suffices to show it for $t=1$.

\medskip\noindent
{\bf (i)} Recalling \eqref{eq:XM01-20241123}, one sees  that
\begin{align*}
|\nabla_x^j p^{(2),\sigma} (1,x)| & \overset{\eqref{eq:XM03}}{=} |\nabla_x^j g^{\sigma} (1,x)| \lesssim_{d,j,\kappa_1} \sum_{k=0}^j |x|^k g^{\sigma} (1,x)\\
& \overset{\eqref{eq:XM03}}{\lesssim}\sum_{k=0}^j |x|^{k} p^{(2)}(1,x) \overset{\eqref{eq:PP04}}{ = } 2^{d/2} p^{(2)}(2,x) \left( \e^{-\frac{|x|^2}{8}}\sum_{k=0}^j |x|^{k} \right)
\\
& \overset{\eqref{eq:PP02}}{\lesssim}  (1\wedge |x|^{-1})^j p^{(2)}(2,x),
\end{align*}
which gives \eqref{eq:GG01}.

\medskip\noindent
{\bf (ii)} For the second inequality, observe that
\begin{align*}\nabla^j p^{(2),\sigma}(1,x)-\nabla^j p^{(2),\sigma}(1,\bar x) = (x-\bar x) \cdot \int_0^1 \nabla^{j+1} p^{(2),\sigma} (1, x + \vartheta (\bar x - x) )\dif \vartheta.
\end{align*}
For $|\bar x - x| \leq1$ and $\vartheta \in [0,1]$, by \eqref{eq:GG01} and  \eqref{eq:IU01-XM}
 , one sees that
\begin{align*}
|\nabla^{j+1} p^{(2),\sigma}| (1, x + \vartheta (\bar x - x) )& \lesssim   (1\wedge |x+\vartheta (\bar x - x)|^{-1})^{j+1} p^{(2)}(2,x + \vartheta (\bar x - x) )\\ &
\lesssim
 p^{(2)}(4,x) \exp\left\{\frac{|\vartheta (\bar x - x)|^2}{4}\right\}
 \lesssim p^{(2)}(4,x),
\end{align*}
which, by the symmetry of $x$ and $\bar x$,  derives that
\begin{align*}
|\nabla^j p^{(2),\sigma}(1,x)-\nabla^j p^{(2),\sigma}(1,\bar x)| \lesssim |x-\bar x| \(p^{(2)}(4,x) \wedge p^{(2)}(4,\bar x)\)
 .
\end{align*}
Hence, combining with \eqref{eq:GG01}, we get the desired estimate.

The proof is complete.
\end{proof}

We are now in a position to give

\begin{proof}[Proof of Theorem \ref{thm:WE01}]
Due to the scaling property \eqref{eq:SS01-X}, we only need to prove this lemma for the case of $s=0$ and $t=1$. In the following, we prove these three inequalities in turn.

\medskip\noindent
{\bf (i)} When $\alpha =2$, by \eqref{eq:GG01} and \eqref{eq:NH01-1}, one sees that
\begin{align*}
|\nabla^j  p^{(2),\sigma}(1,x)|  \lesssim  (1\wedge |x|^{-1})^{j} p^{(2)}(2,x)\lesssim  (1\wedge |x|^{-1})^{j} \varrho_2(2,x) \overset{\eqref{eq:HY00}}{\lesssim} \varrho^{(j)}_2(1,x) .
\end{align*}
Next, define set
$$
J:= \{ k_1,k_2 \in \mN_0 \mid 0\leq k_1\leq j, 0\leq k_2\leq j, \text{ and } 2k_1 - k_2 =j \},
$$
and let $p^{(\alpha)}_{d+k}(t,\tilde x)$ (with $ 0< \alpha < 2$) be the $\alpha$-stable heat kernel in dimension $d+k$ at time $t$ and point $\tilde x \in \mR^{d+k}$.
For any $x \in \mR^d$, letting $\tilde x_k\in \mR^{d+k}$ such that $|\tilde x_k|=|x|$, by \eqref{eq:XM03-1}, and \eqref{eq:XM01-20241123}, and \eqref{eq:XM06}, and definitions, we get that for any $\alpha \in (0,2)$,
\begin{align*}
|\nabla^j p^{(\alpha),\sigma}(1,x) | & = \mE |\nabla^j g^{\sigma,S^{(\alpha)}} (1,x)|  \lesssim  \sum_{(k_1,k_2)\in  J} |x|^{k_2}
 \mE \[ (S_1^{(\alpha)})^{-k_1} p^{(2)}(S_1^{(\alpha)},x)\]\\
& =   \sum_{(k_1,k_2)\in  J}  |x|^{k_2}
\int_{0}^\infty u^{-k_1}   p^{(2)}(u,x) \eta_\alpha (1,u) \dif u\\
&  \lesssim   \sum_{(k_1,k_2)\in  J}  |x|^{k_2}
 \int_{0}^\infty  p^{(2)}_{d+2k_1}(u,\tilde x_{2k_1}) \eta_\alpha (1,u) \dif u \\
& \overset{\eqref{eq:KL03}}{=} \sum_{(k_1,k_2)\in J}  |x|^{k_2}p^{(\alpha)}_{d+2k_1} (1,\tilde x_{2k_1}),
\end{align*}
which, together with \eqref{eq:NH01-1}, implies that for any $\alpha \in [\alpha_0,2)$,
\begin{align*}
|\nabla^j p^{(\alpha),\sigma}(1,x) | & \lesssim\sum_{(k_1,k_2)\in J}  |x|^{k_2}\varrho_\alpha^{(2k_1 )} (1,x) \overset{\eqref{eq:HY00}}{\lesssim}    \varrho^{(j)}_\alpha (1,x).
\end{align*}
Estimate \eqref{eq:NH01} follows.

\medskip\noindent
{\bf (ii)} For the second inequality, observe that
\begin{align}\label{eq:HG01}
\nabla^j p^{(\alpha),\sigma}(1,x)-\nabla^j p^{(\alpha),\sigma}(1,\bar x) = (x-\bar x) \cdot \int_0^1 \nabla^{j+1} p^{(\alpha),\sigma} (1, x + \vartheta (\bar x - x) )\dif \vartheta.
\end{align}
For $|\bar x - x| \leq1$ and $\vartheta \in [0,1]$, by \eqref{eq:NH01} and \eqref{eq:GH01}, one sees that for any $\alpha\in [\alpha_0,2]$,
\begin{align*}
|\nabla^{j+1} p^{(\alpha)}| (1, x + \vartheta (\bar x - x) )& \lesssim   \varrho_\alpha^{(j+1)}(1,x + \vartheta (\bar x - x) )\\ &
\leq 2^{3(d+j+1+\alpha)} \varrho_\alpha^{(j+1)} (1,x) \lesssim \varrho_\alpha^{(j+1)} (1,x),
\end{align*}
which, by the symmetry of $x$ and $\bar x$,  derives that
\begin{align*}
|\nabla^j p^{(\alpha),\sigma}(1,x)-\nabla^j p^{(\alpha),\sigma}(1,\bar x)|
&  \lesssim  |x-\bar x| \left(
\varrho^{(j+1)}_\alpha(1,x) \wedge \varrho^{(j+1)}_\alpha(1,\bar x) \right) \\
& \overset{\eqref{eq:HY00}}{ \lesssim }
 |x-\bar x|
\left( \varrho^{(j)}_\alpha(1,x) \wedge \varrho^{(j)}_\alpha(1,\bar x) \right) .
\end{align*}
Hence, combining with \eqref{eq:NH01}, one sees that
\begin{align*}
|\nabla^j p^{(\alpha),\sigma} (1,x) - \nabla^j p^{(\alpha),\sigma} (1,\bar x)| \lesssim    (1\wedge |x-\bar x|) \(\varrho^{(j)}_\alpha(1,x) + \varrho^{(j)}_\alpha(1,\bar x)\),
\end{align*}
which is exactly the desired estimate \eqref{eq:NH02}.

\medskip\noindent
{\bf (iii)} When $\alpha =2$, \eqref{eq:NH03} follows from \eqref{eq:NH01} and \eqref{eq:HY00}.  As for $\alpha <2$, similar with the proof of~(2.34) in Lemma 2.10 from \cite{MZ22}, one sees that
\begin{align*}
|\Delta^{\alpha/2} \nabla^jp^{(\alpha),\sigma}(1,x)| \lesssim \int_{\mR^d}
\Big(\varrho^{(j)}_\alpha(1,x+z)+\varrho^{(j)}_\alpha(1,x-z)  +\varrho^{(j)}_\alpha(1,x) \Big) (|z|^2\wedge 1) \nu^{(\alpha)}(\dif z).
\end{align*}
Moreover, by \eqref{eq:JH01} and \eqref{eq:JH02}, we have that there is a constant $c_{\alpha_0}$ independent of $\alpha$ such that
$$
\int_{\mR^d} (|z|^2\wedge 1) \nu^{(\alpha)}(\dif z) \leq c_{\alpha_0},
$$
which together with \eqref{eq:GH01} derives that
\begin{align*}
 \int_{\mR^d} \varrho^{(j)}_\alpha(1,x+z)(|z|^2\wedge 1) & \nu^{(\alpha)}(\dif z)
=  \left( \int_{|z|\leq 1\vee \frac{|x|}{2}} + \int_{|z|> 1\vee \frac{|x|}{2}} \right) \varrho^{(j)}_\alpha(1,x+z)(|z|^2\wedge 1) \nu^{(\alpha)}(\dif z) \\
 \lesssim &\, \varrho^{(j)}_\alpha(1,x) \int_{\mR^d}  (|z|^2\wedge 1) \nu^{(\alpha)}(\dif z)  + \int_{|z|> 1\vee \frac{|x|}{2}}  \varrho^{(j)}_\alpha(1,x+z) \nu^{(\alpha)}(\dif z) \\
 \lesssim
 &\, \varrho^{(j)}_\alpha(1,x) +   (1\wedge |x|^{-d-\alpha})  \int_{|z|> 1\vee \frac{|x|}{2}}  \varrho^{(j)}_\alpha(1,x+z)  \dif z \\
  \overset{\eqref{eq:HY00}}{\lesssim } & \varrho_\alpha(1,x)  \left( 1+  \int_{\mR^d} \varrho_\alpha(1,x+z)  \dif z  \right)  \overset{\eqref{eq:LK01}}{\lesssim}\varrho_\alpha(1,x).
\end{align*}
Combining the calculations above, we get the estimate \eqref{eq:NH03}.

The proof is complete.
\end{proof}

From Lemmas \ref{lem:EE01} and \ref{lem:EE02}, we directly have the following corollary, which will be used in the proof of  Lemma \ref{lem:XM80}.

\bc\label{cor:ER01}
Assume that $T>0$ and $\alpha \in [\alpha_0,2]$ with some $\alpha_0\in (0,2)$. For $j \in \mN_0$, there is a constant $c >0$ depending only on $d,T, \alpha_0,j,\kappa_1$ such that, for any $t \in( 0 ,T]$,
\begin{align*}
|\nabla^j  p^{(2),\sigma}(t,x)|  \leq c   (t^{-1/2}\wedge |x|^{-1})^{j} t^{\frac{\alpha-2}{\alpha}}\varrho_\alpha(t,x),
\end{align*}
and
\begin{align*}
|\nabla^j p^{(2),\sigma} (t,x) - \nabla^j p^{(2),\sigma} (t,\bar x)| \leq c    \( (t^{-1/2}|x-\bar x|)\wedge 1 \)t^{\frac{\alpha-2}{\alpha}-\frac{j}{2}}     \( \varrho_\alpha(t,x) + \varrho_\alpha(t,\bar x)\).
\end{align*}
\ec

\subsection{$\alpha$-stability estimates of infinitesimal generators}

Recalling \eqref{eq:XM100}, we define
\begin{align}\label{eq:XM00-20241207}
\widetilde{\mathscr{D}}^{(\alpha),\sigma}_t := \widetilde{\sL}^{(\alpha),\sigma}_t
-\widetilde{\sL}^{(2),\sigma}_t.
\end{align}

The following lemma is analogous to Lemma 2.7 of~\cite{Liu22a}. In this work, we do not use this result directly, but instead follow its proof to establish Lemma \ref{lem:XM80}.

\bl\label{lem-2.4}
 Assume that $\alpha \in [\alpha_0,2)$ with some $\alpha_0 \in (1,2)$. Then for every $f \in C_b^3(\mR^d)$, we have that there is a constant $c = c(d,\alpha_0, \kappa_1)>0$ such that for any $t>0$,
$$
\left | \widetilde{\sD}^{(\alpha),\sigma}_t f(x)\right| \leq c  (2-\alpha)
\left(|\widetilde{\sL}^{(2),\sigma}_t  f(x)|+\|
\sigma(t)\|^3\|\nabla^3 f\|_\infty+
\[\|f\|_\infty\wedge (\|f\|_{\mathrm{Lip}}\|
\sigma(t)\| )\]   \right).
$$
In particular,
$$
\left |\( \Delta^{\alpha/2} -  \Delta\) f(x)\right| \leq c (2-\alpha) \|f\|_{C_b^3},
$$
where $\|f\|_{C_b^3}=  \Sigma_{i=0}^3\|\nabla^i f\|_\infty$.
\el

\begin{proof}
Using the definition of operators, we have  that for any $\epsilon>0$,
\begin{align*}
 \widetilde{\sD}^{(\alpha),\sigma}_t  f(x)
 =&\left( \int_{|z|\leq \epsilon} \[ f(x+\sigma(t) z) - f(x) - \sigma(t) z \cdot \nabla f(x)\] \nu^{(\alpha)} (\dif z) -  \widetilde{\sL}^{(2),\sigma}_t  f(x) \right) \nonumber\\
& + \int_{|z|>\epsilon} \[ f(x+\sigma(t)z) - f(x) \] \nu^{(\alpha)} (\dif z)\nonumber \\
:= & I_1+I_2.
\end{align*}
 For the term $I_1$, adopting the second-order Taylor formula
$$
f(x+z)-f(x)-z\cdot \nabla f(x)=\int_0^1\dif r \int_0^r  \nabla^2 f (x+sz) z \cdot z \dif s,
$$
where $ \nabla^2 f$ is the Hessian matrix of $f$, one sees that
\begin{align*}
I_1 = &  \int_{ |z|\leq \epsilon}\left[\int_0^1\dif r \int_0^r
 \nabla^2 f(x+s \sigma(t)z) (\sigma(t) z) \cdot  \sigma(t) z \dif s\right] \nu^{(\alpha)}(\dif z)-
 \widetilde{\sL}^{(2),\sigma}_t  f(x)  \\
= & \frac{  1}{d}\int_0^1\dif r \int_0^r \int_{ |z|\leq \epsilon}
\widetilde{\sL}^{(2),\sigma}_t  f(x) |z|^2 \nu^{(\alpha)}(\dif z)\dif s
 - \widetilde{\sL}^{(2),\sigma}_t  f(x)\\
 &+  \int_{ |z|\leq \epsilon} \left[ \int_0^1\dif r \int_0^r
\left[\nabla^2 f(x+s \sigma(t)z)-\nabla^2 f(x)\right] (\sigma(t) z)\cdot  \sigma(t)z \dif s \right]\nu^{(\alpha)}(\dif z)  \\
=: &I_{11}+I_{12},
\end{align*}
where  we used the facts,
$$ \int_{ |z|\leq \epsilon} z_i z_j
\nu^{(\alpha)}(\dif z)= \delta_{ij} \frac{1}{d} \int_{ |z|\leq \epsilon}|z|^2
\nu^{(\alpha)}(\dif z)
$$
with $\delta_{ij} = \begin{cases}0, & i \neq j,\\
1,& i =j,
\end{cases}$ being the Kronecker delta, and
\begin{align*}
& \int_{ |z|\leq \epsilon}\left[\int_0^1\dif r \int_0^r  \nabla^2  f(x) (\sigma(t) z) \cdot \sigma(t) z \dif
s\right] \nu^{(\alpha)}(\dif z) \\
= &
\frac{1}{d}\int_0^1\dif r \int_0^r \int_{
|z|\leq \epsilon} \widetilde{\sL}^{(2),\sigma}_t  f(x) |z|^2 \nu^{(\alpha)}(\dif z)\dif s.
\end{align*}
In the following, take $\epsilon=1$ for simplicity. For $I_{11}$, using \eqref{eq:JH01} and \eqref{C-d-alpha}, we obtain that
\begin{align*}
|I_{11}| =  \left| \frac{\omega_{d-1}\cC(d,\alpha)}{ d (2-\alpha)}
\epsilon^{2-\alpha}-1 \right| |\widetilde{\sL}^{(2),\sigma}_t  f(x)|
 \lesssim_d (2-\alpha) |\widetilde{\sL}^{(2),\sigma}_t  f(x)|.
\end{align*}
For $I_{12}$, by \eqref{eq:XM01}, we have
\begin{align*}
|I_{12}|
&
\leq  \int_{ |z|\leq \epsilon}\left|\int_0^1\dif r \int_0^r
 \left[\nabla^2 f(x+s \sigma(t)z)-\nabla^2 f(x)\right] (\sigma(t) z) \cdot \sigma(t) z \dif s\right| \nu^{(\alpha)}(\dif z)\\
 & \lesssim \| \sigma\sigma^*(t)\|^\frac{3}{2}\|\nabla^3 f\|_\infty \int_{ |z|\leq \epsilon} |z|^3 \nu^{(\alpha)}(\dif z)\\
& \lesssim_d
 (2-\alpha)\| \sigma(t)\|^3\|\nabla^3 f\|_\infty,
\end{align*}
where  the last inequality is due to \eqref{eq:JH01}.
For the term $I_2$, by \eqref{eq:JH02}, it is easy to see that
\begin{align*}
|I_2|&  \lesssim \int_{ |z|>\epsilon} \[\|f\|_\infty \wedge (\|f\|_{\mathrm{Lip}}\|
\sigma(t)\|  |z| ) \]\nu^{(\alpha)}(\dif
z)\\
&   \lesssim_{d,\alpha_0} (2-\alpha) \[\|f\|_\infty \wedge (\|f\|_{\mathrm{Lip}}\|
\sigma(t)\|)\] .
\end{align*}
Combining the calculations above,  we finish the proof.
\end{proof}

The following result will be used in the proofs of Corollary \ref{lem:MM01} and Corollary \ref{lem:OI01}.

\bl\label{lem:XM80}
Assume that $T>0$ and $\sigma_1,\sigma$ satisfy the condition \eqref{eq:XM01}. If $\alpha \in [\alpha_0,2)$ with some $  \alpha_0 \in (0,2)$,
then there is a constant $c=c(d,T,\kappa_1,\alpha_0)>0$ such that for
$ 0\leq s< r < t \leq T$,
\begin{align*}
 | \widetilde{\mathscr{D}}^{(\alpha),\sigma_1}_r  p^{(\gamma),\sigma} | (s,t,y) &\leq c (2-\alpha)   (1+  |\ln (t-s)| )(|t-s|^{-1} + |t-s|^{-\frac{2}{\alpha}} ) \varrho_{\alpha} (t-s, y) ,
\end{align*}
where $\gamma\in \{ \alpha,2\}$,  and $\widetilde{\mathscr{D}}^{(\alpha),\sigma}_r$ is defined by \eqref{eq:XM00-20241207}.
\el

\begin{proof} 
For the simplicity of notations, we drop the time variable $r$. Define the first-order difference operator $\delta_z$ by
$$
\delta_{z} h(y) := h(y+z) -h(y).
$$
Following the proof of Lemma \ref{lem-2.4}, a similar argument derives that
\begin{align*}
\widetilde{\sD}^{(\alpha),\sigma_1} f(y) =&\left(\frac{1}{d}\int_{ |z|\leq
\varepsilon}\left[\int_0^1\dif \theta \int_0^\theta  \widetilde{\sL}^{(2),\sigma_1} f(y)|z|^2  \dif
\vartheta\right] \nu^{(\alpha)}(\dif z) -  \widetilde{\mathscr{L}}^{(2),\sigma_1}  f(y) \right) \nonumber\\
&+  \int_{ |z|\leq \varepsilon} \int_0^1\dif \theta \int_0^\theta
\left[\delta_{\vartheta \sigma_1  z} \nabla^2 f(y) \right] (\sigma_1 z)\cdot  \sigma_1  z\dif \vartheta \nu^{(\alpha)}(\dif z)  \\
& + \int_{|z|>\varepsilon} \[ \delta_{\sigma_1 z}f(y)\] \nu^{(\alpha)}
(\dif z)\\
=: & \,   \Lambda^{(1)}_{\alpha} (s,t,y) + \Lambda^{(2)}_{\alpha} (s,t,y)+\Lambda^{(3)}_{\alpha}  (s,t,y).
\end{align*}
In the following, we take $\varepsilon=  |t-s|^{1/\alpha} >0$, and estimate these three terms in turn.

\medskip
\noindent
{\bf (Part 1)} We prove the result for the case of  $f(y) =  p^{(\alpha),\sigma}_{s,t}(y)$ in this part.

\medskip
\noindent(i)  Under the condition \eqref{eq:XM01}, one sees that
\begin{align*}
| \Lambda^{(1)}_{\alpha} (s,t,y)
 |& \overset{\eqref{eq:JH01}}{=}   \xi^{(\alpha)}_{s,t} \times  | \widetilde{\sL}^{(2),\sigma_1}   p^{(\alpha),\sigma}_{s,t}(y)|\\
&\,\overset{\eqref{eq:NH01} }{\lesssim}  \xi_{s,t}^{(\alpha)} \varrho^{(2)}_{\alpha} (t-s,y) \overset{\eqref{eq:HY00} }{\lesssim}  \xi_{s,t}^{(\alpha)}  |t-s|^{-\frac{2}{\alpha}}\varrho_{\alpha} (t-s,y).
\end{align*}
where
\begin{align}
\xi_{s,t}^{(\alpha)} &
 := \left | (t-s)^{\frac{2-\alpha}{\alpha}} \frac{ \omega_{d-1} \cC(d,\alpha)}{d (2-\alpha)} -1 \right|
\nonumber\\
&\overset{\eqref{C-d-alpha}}{ \lesssim} \left| (t-s)^{\frac{2-\alpha}{\alpha}}  - 1\right| + (2-\alpha)
\nonumber\\
&\ \  \lesssim (2-\alpha) (1+|\ln (t-s)|)(|t-s|^\frac{2-\alpha}{\alpha}+ 1)\label{eq:YE00}
\end{align}
since $|t^\eta-1| \leq |\ln t |~ |\eta|~ (t^\eta+1)$ for $t>0$.

\noindent(ii)  Observe that, when $|z|\leq |t-s|^{1/\alpha}$, under the condition \eqref{eq:XM01}, we have
\begin{align*}
 \left| \delta_{\vartheta \sigma_1 z}\nabla^{2}  p^{(\alpha),\sigma}_{s,t} (y)\right| & \overset{\eqref{eq:NH02}}{\lesssim} |t-s|^{-1/\alpha} | \vartheta \sigma_1 z| \( \varrho^{(2)}_\alpha(t-s,y+\vartheta \sigma_1 z) +  \varrho^{(2)}_\alpha(t-s,y)\)\\
 &  \overset{\eqref{eq:GH01}}{\lesssim}  |t-s|^{-1/\alpha} | z| \varrho^{(2)}_\alpha(t-s,y)\\
 &\overset{\eqref{eq:HY00} }{\lesssim} |t-s|^{-3/\alpha} | z| \varrho_\alpha(t-s,y)
\end{align*}
Thus, we have
\begin{align*}
|\Lambda^{(2)}_{\alpha} (s,t,y)|&  \lesssim   |t-s|^{-3/\alpha}
\varrho_{\alpha} (t-s,y)\left( \int_{ |z|\leq \varepsilon}
  |z|^3 \nu^{(\alpha)}(\dif z)\right)\\
  & \overset{\eqref{eq:JH01} }{\lesssim} (2-\alpha)   |t-s|^{-1} \varrho_{\alpha}(t-s,y)
\end{align*}

\noindent(iii) Under the condition \eqref{eq:XM01},  by \eqref{eq:NH02}, we have
$$
|\delta_{\sigma_1 z}p^{(\alpha),\sigma}_{s,t}(y)|\lesssim
  \varrho_{\alpha}(t-s, y+\sigma_1 z)+  \varrho_{\alpha}(t-s, y).
$$
Thus, we obtain that
\begin{align*}
\Lambda^{(3)}_{\alpha}(s,t, y& )   \leq  \left( \int_{ |z|>|t-s|^{1/\alpha}\vee \frac{|y|}{2\kappa_1}} + \int_{ \frac{|y|}{ 2\kappa_1}
\geq |z|>|t-s|^{1/\alpha}} \right)
 \varrho_{\alpha}(t-s, y+\sigma_1 z) \nu^{(\alpha)}(\dif z)\\
 &\qquad + \varrho_{\alpha}(t-s, y)
\int_{ |z|>|t-s|^{1/\alpha}}
\nu^{(\alpha)}(\dif z)\\
&\overset{\eqref{eq:GH01}}{\lesssim}   \int_{
|z|>|t-s|^{1/\alpha}\vee\frac{|y|}{2\kappa_1}}\varrho_{\alpha}(t-s,
y+\sigma_1 z)\nu^{(\alpha)}(\dif z) \\
&\qquad   + 2   \varrho_{\alpha}(t-s,y) \int_{
|z|>|t-s|^{1/\alpha}}    \nu^{(\alpha)}(\dif z)
\\
&\,\, \lesssim (2-\alpha)
|t-s|^{-1}  \varrho_{\alpha}(t-s,y),
\end{align*}
where we used Lemma \ref{lem:ZX01} and \eqref{eq:JH02} in the last inequality. Combining the above calculations and observing \eqref{eq:YE00}, the case of $\gamma=\alpha$ follows.

\medskip\noindent
{\bf (Part II)} In this part, we give the proof of the case of $f(y) =  p^{(2),\sigma}_{s,t}(y)$. The proof is similar to that in {\bf Part I}.

\medskip
\noindent (i) Under the condition \eqref{eq:XM01}, by  Corollary \ref{cor:ER01}, one sees that
\begin{align*}
 | \widetilde{\sL}^{(2),\sigma_1}   p^{(2),\sigma}_{s,t}(y)|  \lesssim  |t-s|^{-1}  |t-s|^{\frac{\alpha-2}{\alpha}}\varrho_\alpha(t-s,x) = |t-s|^{-\frac{2}{\alpha}}\varrho_\alpha(t-s,x) ,
\end{align*}
which derives that
\begin{align*}
| \Lambda^{(1)}_{\alpha} (s,t,y)|  \lesssim \xi_{s,t}^{(\alpha)}  |t-s|^{-\frac{2}{\alpha}}\varrho_\alpha(t-s,x) .
\end{align*}

\noindent(ii) Observe that,  by the conditions \eqref{eq:XM01} and $ |z|\leq |t-s|^{1/\alpha}$, using Corollary \ref{cor:ER01} and \eqref{eq:GH01}, we get
\begin{align*}
 \left| \delta_{\vartheta \sigma_1 z}\nabla^{2}  p^{(2),\sigma}_{s,t} (y)\right|
 & \lesssim |t-s|^{-1/2}  | \vartheta \sigma_1 z| \times    |t-s|^{ -\frac{2}{\alpha}}\varrho_\alpha(t-s,y)\\
 & \lesssim   |t-s|^{-\frac{3}{\alpha}}  | z|   \varrho_\alpha(t-s,y),
\end{align*}
which implies that
\begin{align*}
|\Lambda^{(2)}_{\alpha} (s,t,y)|\lesssim   (2-\alpha)  |t-s|^{-1} \varrho_\alpha(t-s,y).
\end{align*}

\noindent(iii) Under the condition \eqref{eq:XM01},   by Corollary \ref{cor:ER01}, we have
$$
| \delta_{\sigma_1 z}p^{(2),\sigma}_{s,t}(y)|\lesssim |t-s|^{\frac{\alpha-2}{\alpha}}  \( \varrho_\alpha (t-s, y+\sigma_1 z)+\varrho_\alpha (t-s,  y)\).
$$
Thus,
\begin{align*}
\Lambda^{(3)}_{\alpha}(s,t, y& )   \lesssim (2-\alpha) |t-s|^{\frac{\alpha-2}{\alpha}}
|t-s|^{-1}  \varrho_{\alpha}(t-s,y)\\
& \, = (2-\alpha) |t-s|^{-\frac{2}{\alpha}} \varrho_{\alpha}(t-s,y).
\end{align*}
Combining the above calculations and observing \eqref{eq:YE00}, we get the result for $\gamma=2$.

The proof is finished.
\end{proof}

\br 
Following the proof of the lemma above, one sees that: Assume that $\sigma_1,\sigma$ satisfy the condition \eqref{eq:XM01}. If $\alpha \in [\alpha_0,2)$ with some $ \alpha_0 \in (0,2)$,
then there is a constant $  c=c(d,\kappa_1,\alpha_0)>0$ such that  for
$0 \leq s < r< t< +\infty $,
\begin{align*}
 |\widetilde{\mathscr{D}}^{(\alpha),\sigma_1}_r  p^{(\alpha),\sigma} | (s,t,y) &\leq c (2-\alpha)   (1+  |\ln (t-s)| )(|t-s|^{-1}+ |t-s|^{-\frac{2}{\alpha}}) \varrho_{\alpha} (t-s, y) ,
\end{align*}
where  $\widetilde{\mathscr{D}}^{(\alpha),\sigma}_r$ is defined by \eqref{eq:XM00-20241207}.
\er

\bc\label{lem:MM01}
Fix $T>0$. Assume that $\sigma_1,\sigma_2,\sigma$ satisfy the condition \eqref{eq:XM01}. Let $\gamma_1 \in [\alpha_0,2]$ and $\gamma_2 \in [\alpha_0,2)$ with some $ \alpha_0\in (0,2)$. There is a constant $c>0$ depending only on $d,T,\alpha_0,\kappa_1$ such that for all $0\leq s < t   \leq T$,
\begin{align*}
&\quad  \int_{(t+s)/2}^t \int_{\mR^d}  |  p^{(\gamma_1), \sigma_1}  |(r,t,x-y)
| \widetilde{\mathscr{D}}^{(\gamma_2),\sigma}_r p^{(\gamma),\sigma_2}  | (s,r,y) \dif y\dif r\\
& \leq   c (2-\gamma_2)(1+|\ln (t-s) |)
(|t-s|^{-1} +|t-s|^{-\frac{2}{\gamma_2}} )
\sum_{i,j\in\{1,2\}}\varrho_{\gamma_i,\gamma_j} (t-s,x)
\end{align*}
where $ \gamma= \{ \gamma_2,2\}$ and $\widetilde{\mathscr{D}}^{(\gamma_2),\sigma}_r$ is defined by \eqref{eq:XM00-20241207}.
\ec

\begin{proof}
By Lemma \ref{lem:XM80}, one sees that
\begin{align*}
& \qquad  | (\widetilde{\sL}^{(\gamma_2),\sigma}_r  -\widetilde{\sL}^{(2),\sigma}_r)   p^{(\gamma),\sigma_2} (s,r,y) | \\
& \lesssim  (2-\gamma_2)  (1+  |\ln (r-s)| )(|r-s|^{-1}+ |r-s|^{-\frac{2}{\gamma_2}}  )\varrho_{\gamma_2}(r-s, y)  ,
\end{align*}
which, together with  \eqref{eq:NH01} and Lemma \ref{lem-2.4.2}, implies that when $r \in [(t+s)/2,t]$,
\begin{align*}
 &   \int_{\mR^d} | p^{(\gamma_1),\sigma_1}|(r,t,x-y)|(\widetilde{\sL}^{(\gamma_2),\sigma}_r-\widetilde{\sL}^{(2),\sigma}_r)  p^{(\gamma),\sigma_2} | (s,r,y) \dif y \\
   \lesssim  & ~ (2-\gamma_2)  (1+  |\ln (r-s)| )(|r-s|^{-1}+ |r-s|^{-\frac{2}{\gamma_2}} )  \int_{\mR^d}  \varrho_{\gamma_1} (t-r,x-y) \varrho_{\gamma_2} (r-s,y) \dif y\\
\lesssim & ~ (2-\gamma_2)  (1+  |\ln (t-s)| )(|t-s|^{-1}+  |t-s|^{-\frac{2}{\gamma_2}})
 \sum_{i,j \in \{1,2\}}
\varrho_{\gamma_i, \gamma_j}  (t-s,x),
\end{align*}
where we used  \eqref{eq:LK01} in the last inequality. The desired estimate follows.
\end{proof}

\br 
Following the proof of the corollary above, one sees that: Assume that $\sigma_1,\sigma_2,\sigma$ satisfy the condition \eqref{eq:XM01}. Let $\gamma_1 \in [\alpha_0,2]$ and $\gamma_2 \in [\alpha_0,2)$ with some $ \alpha_0\in (0,2)$. There is a constant $c>0$ depending only on $d,\alpha_0,\kappa_1$ such that for all $0\leq s < t < \infty$,
\begin{align*}
&\quad  \int_{(t+s)/2}^t \int_{\mR^d}  |  p^{(\gamma_1), \sigma_1}  |(r,t,x-y)
| \widetilde{\mathscr{D}}^{(\gamma_2),\sigma}_r p^{(\gamma_2), \sigma_2}  | (s,r,y) \dif y\dif r\\
& \leq   c (2-\gamma_2)(1+|\ln (t-s) |)\left(|t-s|^{-1}
+|t-s|^{-\frac{2}{\gamma_2}} \right)
\sum_{i,j\in\{1,2\}}\varrho_{\gamma_i,\gamma_j} (t-s,x)
\end{align*}
where   $\widetilde{\mathscr{D}}^{(\gamma_2),\sigma}_r$ is defined by \eqref{eq:XM00-20241207}.

\er

\subsection{$\alpha$-continuity of heat kernels}

In this subsection we aim to prove the following result about the $\alpha$-continuity of $p^{(\alpha),\sigma}(s,s+t,x)$, which is crucial in this paper.

\begin{theorem}[$\alpha$-continuity]\label{thm:AA01}
Assume that $\alpha \in [\alpha_0,2]$ with some $  \alpha_0\in (0,2)$. Then there is a constant $ c=c(d,T,\alpha_0,\kappa_1)>0$ such that such that for any $ 0\leq s< s+t\leq T$,

\noindent
{\bf (i)} we have that
\begin{align}\label{eq:RT01}
\begin{split}
& |  p^{(\alpha),\sigma}(s,s+t,x)-   p^{(2),\sigma}(s,s+t,x)|\\
 \leq c &\, (2-\alpha)  (1+|\ln t|) (1+
t^{\frac{\alpha-2}{\alpha}})
\sum_{\gamma_1,\gamma_2\in\{\alpha,2\}}\varrho_{\gamma_1,\gamma_2}
(t,x);
\end{split}
\end{align}
{\bf (ii)}  in particular, supposing that
\begin{align}\label{eq:XM02-20241123}
(\sigma\sigma^{\tau})(r) =\begin{bmatrix}
\lambda^{(1)}_r &     &  \\
   &\ddots &  \\
   &   & \lambda^{(d)}_r
\end{bmatrix}_{d\times d} ,
\end{align}
one sees that, for $k=1,2$,
\begin{align}\label{eq:RT01-2}
\begin{split}
&  | \nabla^k p^{(\alpha),\sigma}(s,s+t,x)-   \nabla^k p^{(2),\sigma}(s,s+t,x)|\\
\leq c &\, (2-\alpha) (1+|\ln t|)(1+
t^{\frac{\alpha-2}{\alpha}})\sum_{k_1=1}^k
\sum_{\gamma_1,\gamma_2\in\{\alpha,2\}}|x|^{2k_1 -k} \varrho^{(2k_1)}_{\gamma_1,\gamma_2}
(t,x).
 \end{split}
\end{align}
\end{theorem}

\begin{remark}\label{rem:AA01}
Note that for any $T>0$, there is a constant $c=c(T)>0$ such that for any $\varepsilon \in (0,\alpha)$,
\begin{align*}
1+|\ln t | \leq c \alpha \varepsilon^{-1} {t^{-\varepsilon/\alpha}}.
\end{align*}
Hence, from Theorem \ref{thm:AA01}, one sees that \eqref{eq:RT01} derives that
\begin{align*}
 |  p^{(\alpha),\sigma} -   p^{(2),\sigma}| (s,s+t,x)
 \lesssim  \varepsilon^{-1}(2-\alpha)
t^\frac{\alpha-2-\varepsilon}{\alpha}
\sum_{\gamma_1,\gamma_2\in\{\alpha,2\}}\varrho_{\gamma_1,\gamma_2}
(t,x);
\end{align*}
and from \eqref{eq:RT01-2}, we have that for $k=1,2$,
\begin{align*}
& \quad
  | \nabla^k p^{(\alpha),\sigma}(s,s+t,x)-   \nabla^k p^{(2),\sigma}(s,s+t,x)|\\
  & 
\lesssim  \varepsilon^{-1}(2-\alpha)  t^\frac{\alpha-2-\varepsilon}{\alpha}\sum_{k_1=1}^k
\sum_{\gamma_1,\gamma_2\in\{\alpha,2\}}|x|^{2k_1 -k} \varrho^{(2k_1)}_{\gamma_1,\gamma_2}
(t,x).
\end{align*}
\end{remark}

Now we are in a position to give

\begin{proof}[The proof of Theorem \ref{thm:AA01}]

\noindent{\bf (i)} Notice that the densities are even functions. From the defnition \eqref{eq:XM100},  one see that the  following change of variable holds:
\begin{align*}
\widetilde{\sL}_{t}^{(\alpha),\sigma} f(x) & = \int_{\mR^d} \( f(x+z) -f (x) - z \cdot \nabla f(x) \) \frac{\cC(d,\alpha)\dif z}{|\sigma^{-1}(t) z|^{d+\alpha} |\det \sigma (t)|}\\
& =  \int_{\mR^d} \( f(x+z) -f (x) - z \cdot \nabla f(x) \) \kappa (t,z) \nu^{(\alpha)} (\dif z) ,
\end{align*}
where
$$
\kappa (t,z) =  \frac{|z|^{d+\alpha}}{|\sigma^{-1}(t) z|^{d+\alpha} |\det \sigma (t)|}.
$$
Hence, it is easy to check that for each $t$, the operator $\widetilde{\sL}_{t}^{(\alpha),\sigma}$ is symmetric in the sense of
\begin{align*}
\< f, \widetilde{\sL}_{t}^{(\alpha),\sigma}g\> = \< \widetilde{\sL}_{t}^{(\alpha),\sigma} f, g\>,\ \ \text{for }f,g\in \mathscr{S}.
\end{align*}
Hence, following the  proof of Theorem 2.5 in
\cite{CZ16} (or p.32 in~\cite{CZ18b}), since heat kernels
decay at  infinity, and satisfy \eqref{eq:RR03-X} and \eqref{eq:RR02-X}, we have
\begin{align*}
 p_{s,t}^{(\alpha),\sigma}(x) & - p_{s,t}^{(2),\sigma}(x)=\int_s^t \tfrac{\dif}{\dif r} \<p_{s,r}^{(\alpha),\sigma }(\cdot),  p_{r,t}^{(2),\sigma}(x-\cdot)\>\dif r \\
& = \int_s^t  \left(\<\p_r p_{s,r}^{(\alpha),\sigma}(\cdot),  p_{r,t}^{(2),\sigma}(x-\cdot)\> + \<p_{s,r}^{(\alpha),\sigma}(\cdot),  \p_r p_{r,t}^{(2),\sigma }(x-\cdot)\>\right)\dif r \\
&=  \int^t_{(t+s)/2}\< p_{r,t}^{(2),\sigma}(x-\cdot), (\widetilde{\sL}_r^{(\alpha),\sigma}-\widetilde{\sL}_r^{(2),\sigma})  p_{s,r}^{(\alpha),\sigma}(\cdot) \> \dif r \\
&\quad+  \int_s^{(t+s)/2}  \<p_{s,r}^{(\alpha),\sigma}(\cdot) ,
(\widetilde{\sL}_r^{(\alpha),\sigma}-\widetilde{\sL}_r^{(2),\sigma}) p_{r,t}^{(2),\sigma}(x-\cdot) \> \dif r\\
& := I_1 + I_2
\end{align*}
Moreover, notice that, by change of variables and the scaling property \eqref{eq:SS01-X}, we get that
\begin{align*}
I_2 & = \int_{(t+s)/2}^t  \<p_{s,t+s-r}^{(\alpha),\sigma}(\cdot),
(\widetilde{\sL}_{t+s-r}^{(\alpha),\sigma}-\widetilde{\sL}_{t+s-r}^{(2),\sigma})
 p_{t+s-r,t}^{(2),\sigma}(x-\cdot) \> \dif r\\
& = \int_{(t+s)/2}^t  \<
p_{r,t}^{(\alpha),\sigma_2},
(\widetilde{\sL}_{r}^{(\alpha),\widetilde\sigma}-\widetilde{\sL}_{r}^{(2),\widetilde\sigma})
p_{s,r}^{(2),\sigma_1}(x-\cdot)   \> \dif r,
\end{align*}
where
$$ 
\widetilde{\sigma}(u)= \sigma(t+s -u), \quad
\sigma_1(u)=\sigma(t-r+u),\quad \sigma_2(u) = \sigma (s-r+u).
$$
Applying Corollary \ref{lem:MM01} to $I_1,I_2$, \eqref{eq:RT01} follows. 

\medskip\noindent
{\bf (ii)} Observe that, by definition \eqref{eq:XM01-20241123}, since $A_t^\ell = (A_t^\ell)^\tau $, we have
\begin{align*}
\nabla_x g^{\sigma,\ell}  (t,x) = -\frac{(A_t^\ell)^{-1} x}{2} g^{\sigma,\ell}  (t,x) 
.
\end{align*}
which together with the condition \eqref{eq:XM02-20241123} derives that for each $i=1,2,\cdots,d$,
$$
\nabla_{x_i} g^{\sigma,\ell}(s,s+t,x)=-2\pi x_i
g^{\bar{\sigma},\ell}(s,s+t,\tilde{x}),
$$
where $\tilde{x}\in \R^{d+2}$ is such that $|\tilde{x}|=|x|$ and $g^{\bar{\sigma},\ell}(s,s+t,\tilde{x})$ is the probability density of a Gaussian random variable in dimension $d+2$ with $$(\bar{\sigma} \bar{\sigma}^{\tau})(r)  =\begin{bmatrix}
(\sigma \sigma^\tau)_r & & \\
 &\lambda^{(i)}_r &\\
&  &\lambda^{(i)}_r
\end{bmatrix}_{(d+2) \times (d+2)}.
$$
Hence, by the first equation of formula \eqref{eq:XM03}, and \eqref{eq:XM03-1}, we have that
$$
|\nabla_{x_i}p^{(\alpha),\sigma}-
\nabla_{x_i}p^{(2),\sigma}|(s,s+t,x) \leq 2\pi|x_i|
|p^{(\alpha),\bar{\sigma}}-p^{(2),\bar{\sigma}}|(s,s+t,\tilde{x}),
$$
where $p^{(\alpha),\bar{\sigma}} $ is the transition
densities of a semigroup  in dimension $d + 2$. Consequently, thanks to {\bf (i)} or \eqref{eq:RT01}, and the fact
\begin{align*}
|\tilde{x}|\sum_{\gamma_1,\gamma_2\in\{\alpha,2\}}\varrho^{(2)}_{\gamma_1,\gamma_2} (t,|\tilde{x}|)  = |x|\sum_{\gamma_1,\gamma_2\in\{\alpha,2\}}\varrho^{(2)}_{\gamma_1,\gamma_2} (t,x),
\end{align*}
we get the estimate \eqref{eq:RT01-2} for $k=1$. The proof of  the case $k=2$ is similar, so we omit it here.
\end{proof}

%%%%%%%%%%%%%%%%%%%%%%%%%%%%%%%%%%%%%%%%%%%%%
\section{Densities of SDEs: Proof of Theorem \ref{th-1.1}}\label{sec:th1}
%%%%%%%%%%%%%%%%%%%%%%%%%%%%%%%%%%%%%%%%%%%%%

Throughout this section, we always suppose that $\alpha \in [\alpha_0,2]$ with some $\alpha_0 \in (1,2)$. Consider the following SDE:
$$
\dif X^{(\alpha)} _t=b(X^{(\alpha)}_t) \dif t+ \dif L^{(\alpha)}_t, \quad X^{(\alpha)}
_0=x\in \mathbb{R}^d,
%\sL^{(\alpha)}:=\Delta^{\alpha/2} +  b(x)\cdot \nabla,
$$
where $b$ satisfies the condition ($\textbf{H}_b^\beta$). Our aim is to give some estimates for $\hp^{(\alpha)}_{b}$ and $\hp^{(\alpha)}_{b} - \hp^{(2)}_{b}$ (see Theorem \ref{th-1.1}). Notice that the results in this section are proved for SDEs under ($\mathbf{H}_b^\beta$) and
\begin{align*}
\| \nabla^j b \|_\infty < \infty, \quad \text{for}~ j \in \mN_0.
\end{align*}
Similar to Section 5 of~\cite{MZ22}, by a
standard approximation,  we can extend the bounds of the
estimates under the sole assumptions ($\mathbf{H}_b^\beta$).

\subsection{Convolution inequalities}

For $\eta \in \mathbb{R}$ and $\gamma_1,\gamma_2>0$ and $0\leq s<t\leq T$, let
\begin{equation}\label{eq:phi}
\phi^{(\eta)}_{\gamma_1,\gamma_2} (s,t;x,y) :=
(t-s)^{\eta-1}\varrho_{\gamma_1,\gamma_2} (t-s, \theta^{(1)}_{s,t}(x)-y).
\end{equation}
For simplicity, we denote
$$
\phi_{\gamma}^{(\eta)}:=\phi^{(\eta)}_{\gamma,\gamma}.
$$
Moreover, by Lemma \ref{lem:flow}, if
($\mathbf{H}_b^\beta$) holds, then there is a constant $c=c(T,d,\kappa_0)>0$ such that
\begin{align}\label{eq:SF00}
\phi^{(1)}_{\gamma_1,\gamma_2}  (s,t;x,y) \asymp_c
\varrho_{\gamma_1,\gamma_2}(t-s,x- \theta^{(1)}_{t,s}(y))\asymp_c \varrho_{\gamma_1,\gamma_2}(t-s,\theta^{(1)}_{s,r}(x)- \theta^{(1)}_{t,r}(y)).
\end{align}
 By   \eqref{eq:LK01}, one sees that
\begin{equation}\label{eq:LK001}
  \int_{\R^d}\phi^{(\eta)}_{\gamma_1,\gamma_2}
(s,t;x,y)\dif y \lesssim_{d} (1+T)^{\frac{\gamma_2 \vee \gamma_1}{\gamma_1
\wedge \gamma_2}}
(t-s)^{\eta-\frac{\gamma_2 \vee \gamma_1}{\gamma_1
\wedge \gamma_2}} .
\end{equation}

\bl\label{lem:WR01-2}
Fix $T>0$.
Assume that $\gamma_i \in
[\alpha_0,2]$ with some $\alpha_0\in (0,2)$, $i=1,2,3,4$. Then there is a constant $\hat C=c(T,d,\alpha_0, \kappa_0)>1$ such that for any $(s,t;x,y) \in \D_T$ and $r\in (s,t)$,
\begin{align*}
 \(  \phi^{(1)}_{\gamma_1,\gamma_2}\odot \phi^{(1)}_{\gamma_3,\gamma_4} \)_r (s,t; x,y)
 \leq \hat C \left(  |r-s|^{1-\frac{{\gamma_1 \vee\gamma_2}}{\gamma_1 \wedge\gamma_2}}  +|t-r|^{1- \frac{{\gamma_3 \vee\gamma_4}}{\gamma_3 \wedge\gamma_4}} \right)   \sum_{\substack{i=1,3\\ j =2,4}} \phi^{(1)}_{\gamma_i,\gamma_j  }(s,t;x,y).
\end{align*}
Moreover, if $ \eta_1+1- \tfrac{{\gamma_1 \vee\gamma_2}}{\gamma_1 \wedge\gamma_2}>0,\eta_2+1- \tfrac{{\gamma_3 \vee\gamma_4}}{\gamma_3 \wedge\gamma_4}>0$, then there is a constant $\hat C=c(T,d,\alpha_0,\kappa_0)>1$ such that for any $(s,t;x,y) \in \D_T$,
\begin{align*}
 \( \phi^{(\eta_1)} _{\gamma_1,\gamma_2}\otimes \phi^{(\eta_2)}_{\gamma_3,\gamma_4}\)(s,t;x,y)
 & \leq   \hat C
 {\rm B}\(\eta_1+1-\tfrac{{\gamma_1 \vee\gamma_2}}{\gamma_1 \wedge\gamma_2}, \eta_2+1-\tfrac{{\gamma_3 \vee\gamma_4}}{\gamma_3 \wedge\gamma_4}\) \\
  & \times \sum_{\substack{i=1,3\\ j =2,4}}
\phi_{\gamma_i,\gamma_j} ^{(\eta_1+\eta_2+2-\tfrac{{\gamma_1 \vee\gamma_2}}{\gamma_1 \wedge\gamma_2}-\tfrac{{\gamma_3 \vee\gamma_4}}{\gamma_3 \wedge\gamma_4})}  (s,t;x,y).
\end{align*}
\el

\begin{proof}
By the definitions, Lemma~\ref{lem-2.4.2},  \eqref{eq:SF00}, and \eqref{eq:LK001}, it is easy to check that the first inequality is true. Based on this, by \eqref{eq:Beta} and the non-increasing property of  Beta functions, we get the second one.
\end{proof}

\subsection{Parametrix expansion of densities}

To prove Theorem \ref{th-1.1}, following the idea of~\cite{MZ22, MPZ21}, we use the parametrix method to derive the transition density of SDE. Fix $(\tau,\xi) \in[s, t] \times \R^d$. Consider the following
freezing process:
$$
X^{(\alpha,\tau,\xi)}_{s,t} :=x+\int_s^t b(\theta^{(1)}_{\tau,r}  (\xi)
)\dif r+ \int_s^t   \dif
L^{(\alpha)}_r.
$$
By Lemma 2.8 of~\cite{MZ22}, The density of $X^{(\alpha,\tau,\xi)}_{s,t}$  is given by
 $$
 \bar p^{(\alpha,\tau,\xi)} (s,t;x,y) :=p_{s,t}^{(\alpha)} \bigg(x-y+\int_s^t b(\theta^{(1)}_{\tau,r}  (\xi) ) \dif r \bigg),
$$
where $p^{(\alpha)}_{s,t}$ is the density of $
\int_s^t   \dif L^{(\alpha)}_r$.

Define
\begin{align}\label{eq:df-L}
\sL^{(\alpha)} f(x) :=
\begin{cases}
\int_{\mR^d} \(f(x+   z) - f(x) -  z\cdot \nabla f(x)\) \nu^{(\alpha)} (\dif z), & \alpha \in (1,2),\\
\frac{1}{2}  \Delta f (x), & \alpha=2,
\end{cases}
\end{align}
Due to Lemma 3.1 of~\cite{MZ22}, We have the following Duhamel-type
representation formula:
$$
\hp_{b}^{(\alpha)} (s,t;x,y)= \bar p^{(\alpha,\tau,\xi)}
(s,t;x,y) +  \int_s^t \int_{\mR^d} \hp_{b}^{(\alpha)} (s,r;x,z)
\mathfrak{B}_r^{(\tau,\xi)}  \bar p^{(\alpha,\tau,\xi)}
(r,t;\cdot,y)(z) \dif z \dif r,
$$
where
$$
\mathfrak{B}_s^{(\tau,\xi)}f(x):=\left(b(x)-b(\theta^{(1)}_{\tau,s}  (\xi))
\right)\cdot \nabla f(x) .
$$
In particular, taking $(\tau,\xi)=(t,y)$ and defining
\begin{align}\label{eq:DF02}
\hp_0^{(\alpha)}(s,t;x,y):=\bar p^{(\alpha,t,y)} (s,t;x,y) =
p_{s,t}^{(\alpha)}(x-\theta^{(1)}_{t,s}(y)),
\end{align}
and
\begin{align}\label{eq:SX00}
q_0^{(\alpha)}(s,t;x,y) :=\mathfrak{B}_s^{(t,y)}
\hp_0^{(\alpha)} (s,t;\cdot,y)(x) ,
\end{align}
one sees that
\begin{align*}
\hp_{b}^{(\alpha)} (s,t;x,y)
& = \hp_0^{(\alpha)} (s,t;x,y) +
\int_s^t \int_{\mR^d} \hp_{b}^{(\alpha)} (s,r;x,z)
\mathfrak{B}_r^{(t,y)}  \hp_0^{(\alpha)} (r,t;\cdot,y)(z) \dif
z \dif s\\
& = \hp_0^{(\alpha)} (s,t;x,y) +\hp_{b}^{(\alpha)}
\otimes q_0^{(\alpha)}(s,t;x,y).
\end{align*}
Hence, define recursively  for $n \geq 1$,
$$
 q_n^{(\alpha)}:= q_0^{(\alpha)} \otimes q_{n-1}^{(\alpha)}, \quad q^{(\alpha)}:=\sum_{n=0}^\infty q_n^{(\alpha)},
$$
where the series $q^{(\alpha)}$ are absolutely convergent by Lemma 3.3 of~\cite{MZ22}. Moreover, by iteration and Remark 3.4 of~\cite{MZ22}, we obtain that
\begin{align}\label{eq:DF01}
\hp_{b}^{(\alpha)}=\hp_0^{(\alpha)}+\hp_{b}^{(\alpha)}
\otimes q_0^{(\alpha)}=\hp_0^{(\alpha)}+\sum_{n=0}^\infty
\hp_0^{(\alpha)} \otimes q_n^{(\alpha)}=\hp_0^{(\alpha)}+\hp_0^{(\alpha)}
\otimes q^{(\alpha)}.
\end{align}

\subsection{Estimates of $\hp_0^{(\alpha)}(s,t;x,y)$ and  $|\hp_0^{(\alpha)} -\hp_0^{(2)}|(s,t;x,y)$}

The following result is a direct consequence of  Theorem \ref{thm:WE01}.

\begin{lemma}\label{lem:NH01}
Assume that $\alpha \in [\alpha_0,2]$ with some $\alpha_0\in (0,2)$, and $ j\in\mN_0$. Then there is a $c=c(T,\alpha_0,j,d)>1$, such that for any $(s,t;x,y) \in \D_{+\infty}$,
\begin{align*}
 |\nabla^j \hp_0^{(\alpha)}(s,t;\cdot,y)|(x)\leq c   \varrho^{(j)}_\alpha(t-s,x-\theta^{(1)}_{t,s}(y)).
\end{align*}
\end{lemma}

From this result, we get the following corollary.

\bc\label{cor:TF00-2} Assume that $\alpha \in [\alpha_0,2]$ with
some $\alpha_0\in (1,2)$ and $T>0$. If the condition ($\mathbf{H}_b^\beta$)  holds, then
there is a constant $C_0=c(T,\alpha_0,\Theta)>1$, such that  for any $(s,t;x,y) \in \D_T$,
\begin{align*}
|\mathfrak{B}_s^{(t,y)} \hp_0^{(\alpha)}(s,t;\cdot,y)|(x)\leq C_0
\phi^{(\frac{\alpha+\beta-1}{\alpha})}_{\alpha} (s,t;x,y).
\end{align*}
\ec

\begin{proof}
Recalling the definition \eqref{eq:phi}, by ($\mathbf{H}_b^\beta$) and Lemma \ref{lem:NH01}, for $\beta \in (0,1)$, we have
\begin{align*}
|\mathfrak{B}_s^{(t,y)} \hp_0^{(\alpha)}(s,t;\cdot,y)|(x)& \leq | b(x) - b(\theta^{(1)}_{t,s}(y))| \times | \nabla \hp_0^{(\alpha)}(s,t;\cdot,y)|(x)\\
& \lesssim (|x-\theta^{(1)}_{t,s}(y)|+|x-\theta^{(1)}_{t,s}(y)|^\beta)   \varrho^{(1)}_\alpha(t-s,x-\theta^{(1)}_{t,s}(y))\\
& \overset{\eqref{eq:HY00}}{\lesssim}  \varrho^{(0)}_\alpha(t-s,x-\theta^{(1)}_{t,s}(y)) +  \varrho^{(1-\beta)}_\alpha(t-s,x-\theta^{(1)}_{t,s}(y))\\
&\overset{\eqref{eq:HY00}}{\lesssim}  \phi^{(\frac{\alpha+\beta-1}{\alpha})}_{\alpha} (s,t;x,y),
\end{align*}
where we used the first inequality in \eqref{eq:SF00}. The proof is complete.
\end{proof}

The following lemma is a direct application of Remark \ref{rem:AA01}.

\bl \label{lem:AA01}
Assume that $\alpha \in [\alpha_0,2]$ with some $\alpha_0\in (1,2)$, and $T>0$, and $k=1,2$. Then there is a constant $c=c(T,\alpha_0,k,d)>1$, such that  for any $(s,t;x,y) \in \D_T$ and all $\varepsilon \in (0,\alpha)$,
\begin{align*}
\begin{split}
 &\qquad | \hp_0^{(\alpha)}(s,t;\cdot,y) - \hp_0^{(2)}(s,t;\cdot,y)|(x)\\
 & \leq c\varepsilon^{-1}  (2-\alpha) (t-s)^{ \frac{\alpha-2-\varepsilon}{\alpha}} \sum_{\alpha_1,\alpha_2\in\{\alpha,2\}} \varrho_{\alpha_1,\alpha_2}
(t-s,x-\theta^{(1)}_{t,s}(y)),
 \end{split}
\end{align*}
and
\begin{align*}
\begin{split}
 &\qquad | \nabla^k \hp_0^{(\alpha)}(s,t;\cdot,y) -   \nabla^k \hp_0^{(2)}(s,t;\cdot,y)|(x)\\
 & \leq c\varepsilon^{-1}  (2-\alpha) (t-s)^{ \frac{\alpha-2-\varepsilon}{\alpha}} \sum_{k_1=1}^k
\sum_{\alpha_1,\alpha_2\in\{\alpha,2\}}|x-\theta^{(1)}_{t,s}(y)|^{2k_1 -k} \varrho^{(2k_1)}_{\alpha_1,\alpha_2}
(t-s,x-\theta^{(1)}_{t,s}(y)).
 \end{split}
\end{align*}
\el

Applying this lemma, we have the following corollary.

\bc\label{lem:OI01}
Assume that $\alpha \in [\alpha_0,2]$ with some $\alpha_0\in (1,2)$ and $T>0$. If the condition ($\mathbf{H}_b^\beta$)  holds, then there is a constant $C_1=c(T,\alpha_0,\Theta)>1$ such that  for any $(s,t;x,y) \in \D_T$ and all $\varepsilon \in (0,\alpha)$,
\begin{align*}
|\mathfrak{B}_s^{(t,y)} (\hp_0^{(\alpha)}-\hp_0^{(2)}) (s,t;\cdot,y)|(x)  \leq  C_1 \varepsilon^{-1}  (2-\alpha)
\sum_{\alpha_1,\alpha_2\in\{\alpha,2\}}\phi^{(\frac{1+\beta}{2}+\frac{2\alpha-4-\varepsilon}{\alpha})}_{\alpha_1,\alpha_2}
(s,t;x,y),
\end{align*}
 \ec

\begin{proof}
Thanks to ($\mathbf{H}_b^\beta$), by Lemma \ref{lem:AA01}, we have that
\begin{align*}
|\mathfrak{B}_s^{(t,y)} (\hp_0^{(\alpha)} & -\hp_0^{(2)}) (s,t;\cdot,y)|(x)
  \leq | b(x) - b(\theta^{(1)}_{t,s}(y))| \times |( \nabla \hp_0^{(\alpha)} -\nabla \hp_0^{(2)} ) (s,t;\cdot,y)|(x)\\
& \lesssim  \varepsilon^{-1} (2-\alpha)
(t-s)^{\frac{\alpha-2-\varepsilon}{\alpha}}  (|x-\theta^{(1)}_{t,s}(y)|^2+|x-\theta^{(1)}_{t,s}(y)|^{1+\beta} ) \\
& \qquad \times
\sum_{\alpha_1,\alpha_2\in\{\alpha,2\}}\varrho^{(2)}_{\alpha_1,\alpha_2}
(t-s,x-\theta^{(1)}_{t,s}(y))\\
 & \overset{\eqref{eq:LRW01}}{\lesssim} \varepsilon^{-1} (2-\alpha)
(t-s)^{\frac{1+\beta}{2}+\frac{\alpha-4-\varepsilon}{\alpha}}
 \sum_{\alpha_1,\alpha_2\in\{\alpha,2\}}\varrho_{\alpha_1,\alpha_2}
(t-s,x-\theta^{(1)}_{t,s}(y)),
\end{align*}
which derive the desired estimate since \eqref{eq:phi} and \eqref{eq:SF00}.

The proof is finished.
\end{proof}

\subsection{Estimates of $q^{(\alpha)}(s,t;x,y)$ and $|q^{(\alpha)}-q^{(2)}|(s,t;x,y)$. }

In this subsection, we aim to show Corollary \ref{lem-5.5}. To prove this result, we first prepare some lemmas, and introduce some notations for the later use:
$$
\eta_1 := \frac{1+\beta}{2}+\frac{2(\alpha-2)}{\alpha},\qquad \eta_2 :=\eta_1 +\frac{\alpha-2}{\alpha},\qquad
\eta_i^\varepsilon =\eta_i-\frac\varepsilon\alpha,\quad i=1,2.
$$
Notice that
$$
1 \geq  \eta_0 := \frac{\alpha+\beta-1}{\alpha} \geq \eta_1 \geq \eta_2, \quad \eta_0 >0.
$$

Similar to the proof in Lemma 3.3 of~\cite{MZ22}, it is easy to check that the following lemma holds. We give the proof here for reader's convenience.

\bl\label{lem:YT00}
Assume that $T>0$  and $\alpha \in [\alpha_0,2]$ with some
$\alpha_0\in (1,2)$.
If the condition ($\mathbf{H}_b^\beta$)  holds, then for all $(s,t;x,y) \in
\D_{T}$, and any $n\in \mN_0$,
\begin{align*}
 |q_{n}^{(\alpha)}(s,t;x,y)| \leq C_0^{n+1}  \hat C^{n}  \frac{\(\Gamma(\eta_0)\)^{n+1}}{\Gamma ((n+1)\eta_0)}  \phi^{((n+1)\eta_0)}_{\alpha}(s,t;x,y),
\end{align*}
where $C_0>1$ is the constant in Corollary \ref{cor:TF00-2} and $\hat C>1$ is the constant in Lemma \ref{lem:WR01-2}, and both depend only on $T,\alpha_0,\Theta$.
\el

\begin{proof}
First of all, by Lemma \ref{lem:WR01-2}, we have that for any $n \geq 1$,
\begin{align}\label{eq:HU00}
\left| \phi^{(\eta_0)}_{\alpha}  \otimes \phi^{(n\eta_0)}_{\alpha} \right|(s,t;x,y)
 \leq   \hat C \mathrm{B}\(\eta_0,n\eta_0\)
\phi^{((n+1)\eta_0)}_{\alpha} (s,t;x,y).
\end{align}
For $n=0$, by the definition \eqref{eq:SX00}, using
Corollary \ref{cor:TF00-2}, one sees that
\begin{align*}
|q_0^{(\alpha)}(s,t;x,y)|
\leq \left|  \mathfrak{B}_s^{(t,y)}  \hp_0^{(\alpha)} (s,t;\cdot,y) \right| (x)   \leq C_0  \phi^{(\eta_0)}_{\alpha}  (s,t;x,y),
\end{align*}
which, by \eqref{eq:HU00}, derives that
\begin{align*}
|q_1^{(\alpha)} (s,t;x,y)| \leq C_0^2 \left| \phi^{(\eta_0)}_{\alpha}  \otimes \phi^{(\eta_0)}_{\alpha} \right|(s,t;x,y)
& \leq  C_0^2 \hat C\mathrm{B}\(\eta_0,\eta_0\)
\phi^{(2\eta_0)}_{\alpha} (s,t;x,y).
\end{align*}
By iteration,  and $\mathrm{B}\(\eta_0,n\eta_0\) = \frac{\Gamma (\eta_0) \Gamma( n \eta_0)}{\Gamma((n+1)\eta_0)}$, and \eqref{eq:HU00}, it is easy to check that for any $n \geq 1$,
\begin{align*}
&\quad q_{n+1}^{(\alpha)}(s,t;x,y)
=  \( q_0^{(\alpha)} \otimes q_{n}^{(\alpha)} \) (s,t;x,y)\\
& \leq C_0^{n+2}  \hat C^{n} \frac{\(\Gamma(\eta_0)\)^{n+1}}{\Gamma ((n+1)\eta_0)}  \( \phi^{(\eta_0)}_{\alpha} (s,t;x,y) \) \otimes \( \phi^{((n+1)\eta_0)}_{\alpha}\)(s,t;x,y)\\
& \leq C_0^{n+2}  \hat C^{n+1}  \frac{\(\Gamma(\eta_0)\)^{n+1}}{\Gamma ((n+1)\eta_0)}  {\rm B}(\eta_0,(n+1)\eta_0)
\phi^{((n+2)\eta_0)}_{\alpha}(s,t;x,y)\\
& \overset{\eqref{eq:ZG10}}{=}C_0^{n+2}  \hat C^{n+1}   \frac{\(\Gamma(\eta_0)\)^{n+2}}{\Gamma ((n+2)\eta_0)}  \phi^{((n+2)\eta_0)}_{\alpha}(s,t;x,y).
\end{align*}
The proof is finished.
 \end{proof}

\bl\label{lem:YT01}
Assume that $T>0$  and $\alpha \in [\alpha_0,2]$ with some
$\alpha_0\in (1,2)$. If the condition ($\mathbf{H}_b^\beta$)  holds, then for any $(s,t;x,y) \in \D_T$ and all $\varepsilon \in (0,\alpha)$,
\begin{align}\label{eq:YH00}
|q_0^{(\alpha)} -q_0^{(2)}|(s,t; x,y)  \leq C_1 \varepsilon^{-1} (2-\alpha)
\sum_{\gamma_1,\gamma_2\in\{\alpha,2\}}  \phi^{(\eta_1^\varepsilon )}_{\gamma_1,\gamma_2}
(s,t;x,y),
\end{align}
where $C_1>1$, depending only on $T,\alpha_0,\Theta$, is the same constant as in Corollary \ref{lem:OI01};
and for any $\alpha_0 \in ( \frac{12}{7+\beta} ,2)$ (to ensure $\eta_2>0$) there is a constant $C_2 = c(T,\alpha_0,\Theta)>1$ such that for any $n \geq 1$, and $\varepsilon \in (0,  \alpha\eta_2)$,
$$
|q_n^{(\alpha)} -q_n^{(2)} |(s,t;x,y) \leq C_2^{n+1}\frac{\(\Gamma(\eta_2^\varepsilon)\)^{n+1}}{\Gamma ((n+1)\eta_2^\varepsilon)}
\sum_{\gamma_1,\gamma_2 \in \{ \alpha,2\}} \phi^{(\eta_1^\varepsilon+ n\eta_2^\varepsilon)}_{\gamma_1,\gamma_2} (s,t;x,y).
$$
\el

\begin{proof}
{\bf (i)}  By the definition \eqref{eq:SX00},
\begin{align*}
(q_0^{(\alpha)} - q_0^{(2)})(s,t;x,y)
= \mathfrak{B}_s^{(t,y)}  ( \hp_0^{(\alpha)} -\hp_0^{(2)}
)(s,t;\cdot,y)(x),
\end{align*}
which, combining with Corollary \ref{lem:OI01}, implies the desired result \eqref{eq:YH00}.

\medskip\noindent
{\bf (ii)}  Note that, for any $n\geq 1$,
\begin{align}\label{eq:ZU90}
|q_n^{(\alpha)} -q_n^{(2)}|   & \leq \left | \left(  q_0^{(\alpha)}
-q_0^{(2)} \right)  \otimes q_{n-1}^{(\alpha)}  \right|
+\left |q_0^{(2)} \otimes \left( q_{n-1}^{(\alpha)} - q_{n-1}^{(2)}\right) \right | .
\end{align}

\medskip\noindent
{(Step 1)} First of all, since $1 \geq \eta_0$, following the proof of Lemma \ref{lem:YT00}, it is easy to see that there exist two constants $C_0 ,\hat C >1$, depending only on $d,\alpha_0,\Theta$, such that for any $n \in \mN_0$ and $\gamma \in (0,\eta_0]$,
\begin{align}\label{eq:HU01}
 |q_{n}^{(\alpha)}(s,t;x,y)| \leq (C_0\vee \hat C)^{n+1}    \frac{\(\Gamma(\gamma)\)^{n+1}}{\Gamma ((n+1)\gamma)}  \phi^{((n+1)\gamma)}_{\alpha}(s,t;x,y).
\end{align}
For $ \varepsilon \in (0,\alpha \eta_2)$, we have $ 1 \geq \eta_0 \geq  \eta_1^\varepsilon>0$. Thus, taking $\gamma=\eta_1^\varepsilon$ in \eqref{eq:HU01}, by \eqref{eq:YH00} and Lemma \ref{lem:WR01-2}, for every $n \geq
1$, we have
\begin{align*}
&\quad | (  q_0^{(\alpha)} -q_0^{(2)})  \otimes q_{n-1}^{(\alpha)} | \\
& \leq C_1 (C_0\vee \hat C)^{n}   \varepsilon^{-1}(2-\alpha) \frac{\(\Gamma(\eta_2^\varepsilon)\)^{n }}{\Gamma (n\eta_2^\varepsilon)} \sum_{\gamma_1,\gamma_2 \in \{ \alpha,2\}}
\(\phi^{ (\eta_1^\varepsilon)}_{\gamma_1,\gamma_2}\otimes
\phi^{(n\eta_1^\varepsilon)}_{\alpha} \)   \\
& \leq   4
(1+T)^{\frac{n(2-\alpha)}{\alpha}}
 C_1 (C_0\vee \hat C)^{n+1} \varepsilon^{-1}(2-\alpha) \frac{\(\Gamma(\eta_2^\varepsilon)\)^{n }}{\Gamma (n\eta_2^\varepsilon)}  {\rm B}(\eta_2^\varepsilon,n\eta_1^\varepsilon)
\sum_{\gamma_1,\gamma_2 \in \{ \alpha,2\}}
\phi^{(\eta_1^\varepsilon+n\eta_2^\varepsilon)}_{\gamma_1,\gamma_2}  \\
&\leq C^{n+1} \varepsilon^{-1}(2-\alpha) \frac{\(\Gamma(\eta_2^\varepsilon)\)^{n+1}}{\Gamma ((n+1)\eta_2^\varepsilon)}    \sum_{\gamma_1,\gamma_2 \in \{ \alpha,2\}}
\phi^{(\eta_1^\varepsilon+n\eta_2^\varepsilon)}_{\gamma_1,\gamma_2} ,
\end{align*}
where we used the non-increasing of  Beta functions and \eqref{eq:ZG10}, and the constant
$$
C\geq  4(1+T)^{\frac{2}{\alpha_0}}
 C_1(C_0\vee \hat C) >1
$$
only depends on $T, \alpha_0, \Theta$.

\medskip\noindent
{(Step 2)} Similarly, by the definition \eqref{eq:SX00} and Corollary \ref{cor:TF00-2}, since $1\geq \frac{1+\beta}{2} \geq \eta_0\geq \eta_1^\varepsilon >0$, we have that
\begin{align}\label{eq:HU90}
|q_0^{(2)}(s,t;x,y)|
\leq C_0  \phi^{(\frac{1+\beta}{2})}_{2}  (s,t;x,y)
\leq  (1+T)C_0 \phi^{(\eta_1^\varepsilon)}_{2}  (s,t;x,y).
\end{align}
Hence, since $1\geq \eta_1^\varepsilon \geq \eta_2^\varepsilon>0$, by \eqref{eq:YH00} and  Lemma \ref{lem:WR01-2}, we get
\begin{align*}
\Bigg |q_0^{(2)} \otimes \left( q_0^{(\alpha)}  -q_0 ^{(2)} \right) \Bigg|
& \leq (1+T) C_1   C_0    \varepsilon^{-1}(2-\alpha) \sum_{\gamma_1,\gamma_2 \in \{ \alpha,2\}}
  \( \phi^{(\eta_1^\varepsilon)}_{2}
\otimes
\phi^{(\eta_1^\varepsilon)}_{\gamma_1,\gamma_2} \)
\\
& \leq  4(1+T)^{\frac{2}{\alpha}} C_1 C_0 \hat C \varepsilon^{-1}(2-\alpha) {\mathrm B}(\eta_1^\varepsilon, \eta_2^\varepsilon)\sum_{\gamma_1,\gamma_2 \in \{ \alpha,2\}}
\phi^{(\eta_1^\varepsilon+\eta_2^\varepsilon)}_{\gamma_1,\gamma_2}
% \\ & \leq 4 (1+T)^{\frac{2+\alpha}{\alpha}}  C_1 C_0 \hat C \varepsilon^{-1}(2-\alpha) {\mathrm B}(\eta_2^\varepsilon, \eta_2^\varepsilon)\sum_{\gamma_1,\gamma_2 \in \{ \alpha,2\}} \phi^{(2\eta_2^\varepsilon)}_{\gamma_1,\gamma_2}   \\
%& \leq C^2 \varepsilon^{-1}(2-\alpha) {\mathrm B}(\eta_2^\varepsilon, \eta_2^\varepsilon)\sum_{\gamma_1,\gamma_2 \in \{ \alpha,2\}}  \phi^{(2\eta_2^\varepsilon)}_{\gamma_1,\gamma_2},
\\ &  \leq C^2 \varepsilon^{-1}(2-\alpha) {\mathrm B}(\eta_2^\varepsilon, \eta_2^\varepsilon)\sum_{\gamma_1,\gamma_2 \in \{ \alpha,2\}}  \phi^{(\eta_1^\varepsilon + \eta_2^\varepsilon)}_{\gamma_1,\gamma_2},
\end{align*}
where we used the non-increasing of  Beta functions again. Consequently, due to \eqref{eq:ZU90}, by the result established in  (Step 1), we obtain  the desired inequality for $n =1$, that is
\begin{align*}
|q_1^{(\alpha)} -q_1^{(2)} |(s,t;x,y) \leq C^2 \varepsilon^{-1}(2-\alpha) {\mathrm B}(\eta_2^\varepsilon, \eta_2^\varepsilon)
\sum_{\gamma_1,\gamma_2 \in \{ \alpha,2\}}
  \phi^{(\eta_1^\varepsilon + \eta_2^\varepsilon)}_{\gamma_1,\gamma_2}(s,t;x,y).
\end{align*}
By iteration and Lemma \ref{lem:WR01-2}, observing \eqref{eq:HU90} and combining the result obtained in (Step 1), it is easy to check that for $n \geq 2$,
 \begin{align*}
 &\quad |q_0^{(2)} \otimes ( q_{n-1}^{(\alpha)} - q_{n-1}^{(2)}) |\\
 & \leq (1+T)  C_0  C^{n}\varepsilon^{-1}(2-\alpha)
 \frac{\(\Gamma(\eta_2^\varepsilon)\)^{n}}{\Gamma (n\eta_2^\varepsilon)}    \sum_{\gamma_1,\gamma_2 \in \{ \alpha,2\}}
  \( \phi^{(\eta_1^\varepsilon)}_{2}
\otimes
\phi^{(\eta_1^\varepsilon+(n-1)\eta_2^\varepsilon)}_{\gamma_1,\gamma_2} \) \\
& \leq 4(1+T)^{\frac{2}{\alpha}}  C_0 \hat C C^{n}\varepsilon^{-1}(2-\alpha)
 \frac{\(\Gamma(\eta_2^\varepsilon)\)^{n}}{\Gamma (n\eta_2^\varepsilon)}
 \mathrm{B}(\eta_1^\varepsilon, n \eta_2^\varepsilon)\sum_{\gamma_1,\gamma_2 \in \{ \alpha,2\}} \phi^{(\eta_1^\varepsilon+n\eta_2^\varepsilon)}_{\gamma_1,\gamma_2} \\
 & \leq C^{n+1}\frac{\(\Gamma(\eta_2^\varepsilon)\)^{n+1}}{\Gamma ((n+1)\eta_2^\varepsilon)}
\sum_{\gamma_1,\gamma_2 \in \{ \alpha,2\}} \phi^{(\eta_1^\varepsilon+n\eta_2^\varepsilon)}_{\gamma_1,\gamma_2},
\end{align*}
which completes the proof.
\end{proof}

\begin{corollary}\label{lem-5.5} Assume that $T>0$  and $\alpha \in [\alpha_0,2]$ with some
$\alpha_0\in (1,2)$.
If the condition ($\mathbf{H}_b^\beta$)  holds, then there is a constant $c=
c(T,\alpha_0,\Theta)>1$ such that for all $(s,t;x,y) \in
\D_{T}$,
$$
q^{(\alpha)}(s,t;x,y) \leq c   \phi^{(\eta_0)}_{\alpha}  (s,t;x,y);
$$
and  for any $\alpha_0 \in ( \frac{12}{7} ,2)$ (to ensure $\eta_2>0$) and $\varepsilon= \frac{\alpha \beta}{2}$,
$$
|q^{(\alpha)}-q^{(2)}|(s,t;x,y) \leq c \varepsilon^{-1} (2-\alpha)
\sum_{\gamma_1,\gamma_2\in\{\alpha,2\}}
  \phi^{(\eta_1^\varepsilon)}_{\gamma_1,\gamma_2}(s,t;x,y).
$$
\end{corollary}

\begin{proof}
Due to the Stirling formula,  we have
that for any $a>0, b>0$,
$ \sum_{n=0}^\infty   \frac{ b^n}{\Gamma ((n+1)a )}   <\infty.$
Hence, observing that $0< \frac{\alpha_0 + \beta -1}{\alpha_0} \leq \eta_0\leq 1$, by Lemma \ref{lem:YT00}, we have
\begin{align*}
q^{(\alpha)}(s,t;x,y) =\sum_{n=0}^\infty q_n^{(\alpha)}(s,t;x,y)& \leq
\sum_{n=0}^\infty   \frac{\(\Gamma(\eta_0) C_0\hat C \)^{n+1}}{\Gamma ((n+1)\eta_0)}  \phi^{((n+1)\eta_0)}_{\alpha}(s,t;x,y)\\
&\leq \Gamma(\eta_0)C_0\hat C
\sum_{n=0}^\infty   \frac{\(\Gamma(\eta_0) C_0\hat C |t-s|^{\eta_0}\)^{n}}{\Gamma ((n+1)\eta_0)}  \phi^{(\eta_0)}_{\alpha}(s,t;x,y)\\
& \lesssim  \phi^{(\eta_0)}_{\alpha}(s,t;x,y),
\end{align*}
where we employed fact that the boundedness of Gamma's function on any closed interval, along with its monotonic increasing property on $(2, +\infty)$. Similarly, applying Lemma \ref{lem:YT01} with $\varepsilon=\frac{\alpha \beta}{2}$, we have $\eta_2^\varepsilon=\frac{7\alpha-12}{2\alpha}$ with  $0< \frac{7\alpha_0-12}{2\alpha_0}\leq \eta_2^\varepsilon \leq 1$ and
\begin{align*}
|q^{(\alpha)}-&q^{(2)} |(s,t;x,y)
  \leq \sum_{n=0}^\infty | q_n^{(\alpha)} -q_n^{(2)}| (s,t;x,y)\\
& \leq \varepsilon^{-1}(2-\alpha) \sum_{n=0}^\infty
 \frac{\(\Gamma(\eta_2^\varepsilon)C_2\)^{n+1}}{\Gamma ((n+1)\eta_2^\varepsilon)}   \sum_{\gamma_1,\gamma_2\in\{\alpha,2\}}
\phi^{(\eta_1^\varepsilon+ n\eta_2^\varepsilon)}_{\gamma_1,\gamma_2} (s,t;x,y)\\
& \leq  \varepsilon^{-1}(2-\alpha)  C_2 \Gamma(\eta_2^\varepsilon)
\sum_{n=0}^\infty
 \frac{\(\Gamma(\eta_2^\varepsilon) C_2 T^{\eta_2^{\varepsilon} }\)^{n}}{\Gamma ((n+1)\eta_2^\varepsilon)}   \sum_{\gamma_1,\gamma_2\in\{\alpha,2\}}
\phi^{(\eta_1^\varepsilon)}_{\gamma_1,\gamma_2} (s,t;x,y)\\
& \lesssim \varepsilon^{-1} (2-\alpha)
\sum_{\gamma_1,\gamma_2\in\{\alpha,2\}}
  \phi^{(\eta_1^\varepsilon)}_{\gamma_1,\gamma_2}(s,t;x,y).
\end{align*}
The desired estimates are established.
\end{proof}

\subsection{Proof of Theorem \ref{th-1.1}}
(1) Recalling \eqref{eq:DF01} and the definition \eqref{eq:phi},
by Lemma \ref{lem:NH01} and Corollary \ref{lem-5.5}, we get
\begin{align*}
\hp_{b}^{(\alpha)} & =  \hp_0^{(\alpha)} +\hp_0^{(\alpha)}
\otimes q^{(\alpha)}
 \lesssim   \(  \phi^{(1)}_{\alpha}  +  \phi^{(1)}_{\alpha}\otimes  \phi^{(\eta_0)}_{\alpha}\)\\
&\lesssim    \phi^{(1)}_{\alpha} +  \mathrm{B}\(1,\eta_0\)  \phi^{(1+\eta_0)}_{\alpha}
\lesssim    \phi^{(1)}_{\alpha} ,
\end{align*}
where we used Lemma \ref{lem:WR01-2} and the fact $1\geq \eta_0 \geq \frac{\alpha_0+\beta-1}{\alpha_0}>0$.

\medskip\noindent(2)
Using \eqref{eq:DF01} again, we have that
$$
|\hp_{b}^{(\alpha)}-\hp_{b}^{(2)}| \leq
|\hp_0^{(\alpha)}-\hp_0^{(2)}|+|(\hp_0^{(\alpha)}-\hp_0^{(2)})\otimes q^{(2)}|+|\hp_0^{(\alpha)} \otimes (q^{(\alpha)}-
q^{(2)})|.
$$
By Lemma \ref{lem:AA01} and the definition \eqref{eq:phi}, one sees that
\begin{align*}
|\hp_0^{(\alpha)}-\hp_0^{(2)}|
 \lesssim \varepsilon^{-1} (2-\alpha) \sum_{\gamma_1,\gamma_2\in\{\alpha,2\}} \phi^{(\frac{
2\alpha-2}{\alpha}-\frac{\varepsilon}{\alpha})}_{\gamma_1,\gamma_2}
% \leq (1+T)^{\frac{4-\alpha_0}{\alpha_0}-\frac{1+\beta}{2}} \varepsilon^{-1} (2-\alpha) \sum_{\gamma_1,\gamma_2\in\{\alpha,2\}} \phi^{(\frac{2\alpha-2}{\alpha}-\frac{\varepsilon}{\alpha})}_{\gamma_1,\gamma_2} ,
\end{align*}
which, together with Corollary \ref{lem-5.5} and Lemma \ref{lem:WR01-2}, derives that
\begin{align*}
& \quad |(\hp_0^{(\alpha)}-\hp_0^{(2)})\otimes q^{(2)}|\\
&\lesssim \varepsilon^{-1}  (2-\alpha)  \sum_{\gamma_1,\gamma_2\in\{\alpha,2\}}\left( \phi^{(\frac{ 2\alpha-2}{\alpha} - \frac{\varepsilon}{\alpha})}_{\gamma_1,\gamma_2}   \otimes  \phi^{(\frac{ \beta+1}{2})}_{2} \right) \\
& \lesssim (1+T)^{\frac{2-\alpha}{\alpha}} \mathrm{B}(\tfrac{ 3(\alpha-2)}{\alpha} - \tfrac{\varepsilon}{\alpha}+\tfrac{2}{\alpha},\tfrac{1+\beta}{2})\varepsilon^{-1}(2-\alpha)
\sum_{\gamma_1,\gamma_2\in\{\alpha,2\}}\phi^{(\eta_2^\varepsilon+\frac{2}{\alpha})}_{\gamma_1,\gamma_2}
\\
 &\lesssim \varepsilon^{-1} (2-\alpha)
\sum_{\gamma_1,\gamma_2\in\{\alpha,2\}}\phi^{(\eta_2^\varepsilon)}_{\gamma_1,\gamma_2}.
\end{align*}
Due to  Lemma \ref{lem:NH01}, Corollary \ref{lem-5.5}, and Lemma \ref{lem:WR01-2},
we get that for $ \alpha_0 > \frac{12}{7}$ and $\varepsilon =\frac{\alpha\beta}{2}$,
\begin{align*}
|\hp_0^{(\alpha)} \otimes (q^{(\alpha)}- q^{(2)})|
&  \lesssim \varepsilon^{-1} (2-\alpha)
\sum_{\gamma_1,\gamma_2\in\{\alpha,2\}} \( \phi_\alpha^{(1)} \otimes  \phi^{(\eta_1^\varepsilon)}_{\gamma_1,\gamma_2}\) \\
&\lesssim (1+T)^{\frac{2-\alpha}{\alpha}} \mathrm{B}(1,\eta_2^\varepsilon )\varepsilon^{-1}(2-\alpha)
\sum_{\gamma_1,\gamma_2\in\{\alpha,2\}}
 \phi^{(\eta_2^\varepsilon+1)}_{\gamma_1,\gamma_2}\\
 & \lesssim \varepsilon^{-1}(2-\alpha)
\sum_{\gamma_1,\gamma_2\in\{\alpha,2\}}
 \phi^{(\eta_2^\varepsilon)}_{\gamma_1,\gamma_2},
\end{align*}
where we used  fact that $\eta_2^\varepsilon=\frac{7\alpha-12}{2\alpha}\geq  \frac{7\alpha_0-12}{2\alpha_0}>0 $ and  the non-increasing of  Beta functions again. Combining the calculations above and noticing that
$$
\frac{
2\alpha-2}{\alpha}-\frac{\varepsilon}{\alpha} - \eta_2^\varepsilon = \frac{4-\alpha}{\alpha} - \frac{1+\beta}{2}> 0,
$$
we get the desired result.

%%%%%%%%%%%%%%%%%%%%%%%%%%%%%%%%%%%%%%%%%%%%%%%%
\section{Optimal convergence rate: Proof of Theorem \ref{th-6.1}}\label{sec:last}
%%%%%%%%%%%%%%%%%%%%%%%%%%%%%%%%%%%%%%%%%%%%%%%%

Throughout this section, we assume that $\alpha \in (0,2]$ and the condiction
$(\mathbf{H}_b^\beta)$ holds.  To study
invariant measures, we make an additional dissipative assumption $(\mathbf{H}_b^{\rm diss})$ on the drift coefficients. By Theorem 1.2 of~\cite{ZZ23}, under our assumptions, there exists a
unique invariant measure $\mu^{(\alpha)}$ of the Markov transition
semigroup $\{P^{(\alpha)}_t\} _{t\geq 0}$ (or the infinitesimal generator ${\sL}^{(\alpha)} + b \cdot \nabla$) associated with equation
\eqref{e1}, where $\sL^{(\alpha)}$  is given by \eqref{eq:df-L}.

\begin{proposition}[Contraction property]\label{cor-6.00}
For every $p\in(0,\alpha)$ and any probability measure $\tilde{\mu}$, there
exists $t_0 > 0$ such that for every $t\geq t_0$,
$$
\| {P_t^{(\alpha)}}^* (\tilde{\mu}-  \mu )\|_{\var,V_p}\leq \frac{1}{2} \|
\tilde{\mu}- \mu \|_{\var,V_p},
$$
where  $V_p(x)=1+|x|^p$.
\end{proposition}

%The time $t_0$ is dependet of $alpha$ or $alpha_0$. But We only use the case of 2.

\begin{proof}
We only prove the case of $\alpha <2$, since the case of $\alpha=2$ is similar. Due to  Eq. (24) of~\cite{ZZ23}, there exist $K_1,K_2 >0$ such that for all $x\in \R^d$,
 \begin{equation}\label{eq:Harris-1}
({\sL}^{(\alpha)}+ b\cdot \nabla) V_p(x) \leq -K_1 V_p(x)+K_2.
 \end{equation}
 On the other hand, we have
\begin{align*}
\| {P_t^{(\alpha)}}^* (\delta_x- \delta_y) \|_{\var}
&=\sup_{\| h \|_\infty \leq 1} \bigg| \EX \[h(X^{(\alpha)}_t(x))-h(X^{(\alpha)}_t(y))\] \bigg| \\
&=\sup_{\| h\|_\infty \leq 1} \left| \int_{\R^d} h(z)(\hp_{b}^{(\alpha)}(0,t;x,z)-\hp_{b}^{(\alpha)}(0,t;y,z)) \dif z \right| \\
&\leq \int_{\R^d} \left| \hp_{b}^{(\alpha)}(0,t;x,z)-\hp_{b}^{(\alpha)}(0,t;y,z) \right |\dif z
\lesssim t^{-\frac{1}{\alpha}}|x-y|,
\end{align*}
where, in the last inequality, we used Eq. (1.25) from~\cite{MZ22}, i.e. $|\nabla_x \log \hp_{b}^{(\alpha)}(0,t;x,y)| \lesssim t^{-\frac{1}{\alpha}}.$
Thus, for every $R>2$, $\varepsilon \in(0,1)$,  we choose  $t>0$ large enough such that for all  $(x,y) \in  \{(x,y) \mid V_p(x)+V_p(y)\leq R\}$,
 \begin{equation}\label{eq:Harris-2}
\|{P_t^{(\alpha)}}^*(\delta_x- \delta_y) \|_{\var}   \leq 2\varepsilon.
 \end{equation}
Due to the Harris theorem (cf. Theorem 5.2 in~\cite{CM23}), the proof is complete by using \eqref{eq:Harris-1} and \eqref{eq:Harris-2}.
\end{proof}

We also need the following  result about uniform bound for moments of invariant
measures $\mu^{(\alpha)}$.

\begin{lemma}\label{lem-6.0} For all $  \gamma \in ( 0, 2)$,
\begin{equation*}
\sup_{\alpha \in [\frac{\gamma+2}{2}, 2]}\int  |x|^\gamma
\mu^{(\alpha)}(\dif x) < \infty.
\end{equation*}
\end{lemma}

\begin{proof}
Let $f(x):=(| x |^2 +1)^\frac{\gamma}{2}$ with $\gamma \in (0,2)$. It is easy to check that $\nabla f(x)=\gamma
x(|x|^2 +1)^{\frac{\gamma}{2}-1}$, and then by $(\mathbf{H}_b^{\rm diss})$,
$$
(b\cdot\nabla) f(x) \leq
2(|x|^2+1)^{\frac{\gamma}{2}-1} (-c_0|x|^{2+r}+c_1)\leq 2c_1-2c_0
|x|^{r+\gamma}.
$$
Following the calculations of equations (3.4)-(3.5) in~\cite{DXZ11}, one sees that
\begin{align*}
\left| f(x+   z) - f(x) - z \1_{|z|\leq
1} \cdot \nabla f(x)\right| \lesssim |z|^2 \1_{|z|\leq
1}+ |z|^\gamma \1_{|z|>
1}.
\end{align*}
Furthermore, in view of \eqref{un-alpha-0}, there is a constant $c>1$ independent of $\alpha$ such that  for every $\gamma \in (0,2)$,
\begin{align*}
 ({\sL}^{(\alpha)} + b \cdot \nabla ) f(x) 
 & \lesssim \sup_{\alpha \in [\frac{\gamma+2}{2},2)}\int_{\mR^d} ( |z|^2 \1_{|z|\leq
1}+ |z|^\gamma \1_{|z|>
1})
\nu^{(\alpha)}(\dif z) + 1-|x|^{r+\gamma}\\
& \leq  2c c_0  - 2c_0 |x |^{r+\gamma}.
\end{align*}
On the other hand, due to the invariance, we have
\begin{equation*}
\int_{\mR^d} \( {\sL}^{(\alpha)} + b \cdot \nabla \) f(x)  \mu^{(\alpha)}(\dif x)=0.
\end{equation*}
Hence, combining the calculations above, we obtain that
$$
\int_{\mR^d}  |x |^{r+\gamma}  \mu^{(\alpha)}(\dif x) \leq c, \quad \text{with} \quad r\geq 0,
$$
and then
$$
\int_{\mR^d}  |x |^{\gamma}  \mu^{(\alpha)}(\dif x) \leq c.
$$
 A similar argument yields the case of $\alpha=2$. The proof is finished.
\end{proof}

\begin{corollary}\label{cor-6.1}
For every  $t>0$ and $\alpha \in [\alpha_0,2]$ withs some $\alpha_0\in (\frac{12}{7},2)$, there is a constant  $c=
c( t,\alpha_0, p, \widetilde \Theta)>0$   such that
$$
 \| ( {P_{t}^{(\alpha)}}^*  - {P_{t}^{(2)}}^*) \mu^{(\alpha)}
 \|_{\var,V_p}\leq c (2-\alpha),
$$
where  $V_p(x)=1+|x|^p$ with $ p\in(0,\alpha)$.
\end{corollary}

\begin{proof}
By Theorem~\ref{th-1.1}, Lemma \ref{lem:flow}, and \eqref{eq:LK01}, we get that
\begin{align*}
\|{P_{t}^{(\alpha)}}^* \delta_x& -{P_{t}^{(2)}}^* \delta_x \|_{\var,V_p} \leq
\int_{\mR^d} (1+|y|^p)|\hp_{b}^{(\alpha)}-\hp_{b}^{(2)}|(0,t;x,y) \dif y
 \\
& \lesssim (2-\alpha) t^{\tfrac{
7\alpha-12}{2\alpha}-1} \sum_{\gamma_1,\gamma_2\in\{\alpha,2\}} \int_{\mR^d} (1+|y|^p)
\varrho_{\gamma_1,\gamma_2}
(t, \theta^{(1)}_{0,t}(x)-y) \dif y\\
& \lesssim
(2-\alpha) (1+|x|^p),
\end{align*}
which implies the desired result by Lemma \ref{lem-6.0}.
\end{proof}

Now we are in the position to give the

\begin{proof}[Proof of Theorem \ref{th-6.1}]
Let $\dis$ be $ \|\cdot\|_{\var,V_p}$ with $p \in (0,\alpha)$.
Due to the invariance of measures $\mu^{(\alpha)}$ and $\mu^{(2)}$, we
have that for $t>0$,
\begin{align*}
\dis(\mu^{(\alpha)},& \mu^{(2)}) =\dis({P_{t}^{(\alpha)}}^*
\mu^{(\alpha)},{P_{t}^{(2)}}^* \mu^{(2)}) \\
& \leq \dis({P_{t}^{(\alpha)}}^*
\mu^{(\alpha)},{P_{t}^{(2)}}^* \mu^{(\alpha)})
+\dis({P_{t}^{(2)}}^*\mu^{(\alpha)},{P_{t}^{(2)}}^* \mu^{(2)}),
\end{align*}
which together with Proposition \ref{cor-6.00} derives that there exists $t_0>0$ such that
$$
\dis(\mu^{(\alpha)},\mu^{(2)}) \leq 2 \dis({P_{t_0}^{(\alpha)}}^*
\mu^{(\alpha)},{P_{t_0}^{(2)}}^* \mu^{(\alpha)}) \lesssim (2-\alpha),
$$
where we used Corollary \ref{cor-6.1} in the last inequality. The proof is finished.
\end{proof}

%%%%%%%%%%%%%%%%%%%%%%%%

%%%%%%%
%APPENDIX
%%%%%%%

%\begin{appendix}

%\section{}

%\end{appendix}

\subsection*{Acknowledgments}
We thank Zimo Hao (Beijing Institute of Technology) and Guohuan Zhao (Academy of Mathematics and Systems Science of CAS) for their insightful discussions and valuable corrections.

\subsection*{Fundings}
Mingyan Wu was supported by the National Natural Science Foundation of China (NSFC) Grant No. 12201227. Xianming Liu was supported in part by the NSFC Grant No. 11971186.

%\addcontentsline{toc}{section}{References}

\end{document}